\documentclass[12pt]{article}

\usepackage[utf8]{inputenc}
\usepackage[english]{babel}
\usepackage[T1]{fontenc}
\usepackage{graphicx}
\usepackage{graphics}
\usepackage{amsfonts}
\usepackage{amssymb}
\usepackage{amsmath}
\usepackage{amsthm}
\usepackage{mathrsfs}
\usepackage{array}
\usepackage[cm]{fullpage}
\usepackage{verbatim}
\usepackage[numbers]{natbib}
\usepackage{float}

\usepackage{color}

\newtheorem{thm}{Theorem}[section]

\newtheorem{lemma}[thm]{Lemma}
\newtheorem{proposition}[thm]{Proposition}
\newtheorem{corol}[thm]{Corollary}

\theoremstyle{definition}

\newtheorem{example}[thm]{Example}
\newtheorem{remark}[thm]{Remark}

\newcommand{\C}{{\mathbb C}}

\newcommand{\N}{{\mathbb N}}

\newcommand{\R}{{\mathbb R}}

\newcommand{\KK}{\mathchoice{\text{\large$\mathfrak{K}$}}
                            {\text{\large$\mathfrak{K}$}}
                            {\text{\normalsize$\mathfrak{K}$}}
                            {\text{$\mathfrak{K}$}}}

\newcommand{\eps}{\varepsilon}

\newcommand{\rrarrow}{\rightrightarrows}
\usepackage{mathrsfs}   

\newcommand{\Ecal}{\mathscr{E}}

\newcommand{\Rcal}{\mathscr{R}}

\newcommand{\ut}{\tilde{u}}

\newcommand{\zt}{\tilde{z}}

\newcommand{\At}{\tilde{A}}

\newcommand{\sbar}{\bar{s}}
\newcommand{\ub}{\bar{u}}
\newcommand{\vb}{\bar{v}}
\newcommand{\xb}{\bar{x}}
\newcommand{\yb}{\bar{y}}
\newcommand{\zb}{\bar{z}}

\newcommand{\bi}{\begin{itemize}}
\newcommand{\ei}{\end{itemize}}
\newcommand{\bpm}{\begin{pmatrix}}
\newcommand{\epm}{\end{pmatrix}}

\newcommand{\norm}[1]{|#1|}

\newcommand{\ssoucin}[2]{\langle #1,#2\rangle}

\newcommand{\epi}{\operatorname{epi}}
\newcommand{\Nc}{\operatorname{N}}
\newcommand{\Tc}{\operatorname{T}}

\newcommand{\gph}{\operatorname{gph}}
\newcommand{\cl}{\operatorname{cl}}

\newcommand{\Limsup}{\operatorname*{Limsup}}

\newcommand\DT[1]{\mathchoice
                 {{\buildrel{\hspace*{.1em}\text{\LARGE.}}\over{#1}}}
                 {{\buildrel{\hspace*{.1em}\text{\Large.}}\over{#1}}}
                 {{\buildrel{\hspace*{.1em}\text{\large.}}\over{#1}}}
                 {{\buildrel{\hspace*{.1em}\text{\large.}}\over{#1}}}}
\renewcommand\P{\pi}
\newcommand\Pb{\bar{\P}}
\newcommand\GD{\Gamma_{\scriptscriptstyle\mathrm{D}}}
\newcommand\GC{\Gamma_{\scriptscriptstyle\mathrm{C}}}
\newcommand{\WT}[1]{\mathchoice
     {\text{\small$\widetilde{\text{\normalsize$#1$}}\hspace*{.03em}$}}
                    {\text{\small$\widetilde{\text{\normalsize$#1$}}$}}
                    {\widetilde{#1\hspace*{-.02em}}\hspace*{.03em}}
                    {\tilde{#1}}}

\newcommand\bigtimes{\text{\Large$\times$}}

\newcommand{\nablaS}{\nabla_{\scriptscriptstyle\textrm{\hspace*{-.2em}S}}^{}}
\newcommand{\divS}{\mathrm{div}_{\scriptscriptstyle\textrm{\hspace*{-.1em}S}}^{}}
\newcommand{\uC}{u_{\scriptscriptstyle\textrm{C}}}
\newcommand{\uD}{u_{\scriptscriptstyle\textrm{D}}}
\newcommand{\uF}{u_{\scriptscriptstyle\textrm{F}}}
\newcommand{\uT}{u_{\scriptscriptstyle\textrm{T}}}
\newcommand{\uN}{u_{\scriptscriptstyle\textrm{N}}}
\newcommand{\wD}{w_{\scriptscriptstyle\textrm{D}}}
\newcommand{\CC}{_{\scriptscriptstyle\textrm{CC}}}
\newcommand{\CD}{_{\scriptscriptstyle\textrm{CD}}}
\newcommand{\FD}{_{\scriptscriptstyle\textrm{FD}}}
\newcommand{\CF}{_{\scriptscriptstyle\textrm{CF}}}
\newcommand{\FC}{_{\scriptscriptstyle\textrm{FC}}}
\newcommand{\FF}{_{\scriptscriptstyle\textrm{FF}}}
\newcommand{\DC}{_{\scriptscriptstyle\textrm{DC}}}
\newcommand{\DD}{_{\scriptscriptstyle\textrm{DD}}}
\newcommand{\DF}{_{\scriptscriptstyle\textrm{DF}}}
\newcommand{\kappaN}{\kappa_{\scriptscriptstyle\textrm{N}}^{}}
\newcommand{\kappaT}{\kappa_{\scriptscriptstyle\textrm{T}}^{}}
\newcommand{\alphaF}{\alpha_{\scriptscriptstyle\textrm{F}}^{}}
\newcommand{\dx}{\,\mathrm{d}x}
\newcommand{\dt}{\,\mathrm{d}t}
\newcommand{\dS}{\,\mathrm{d}S}

\newcommand{\coder}{\operatorname{D}^*}
\newcommand{\partition}{\tilde{Q}}
\newcommand{\partitionbig}{Q}

\begin{document}

\vspace*{1em}

\begin{center}
\Large\bfseries
Identification of some nonsmooth evolution systems with illustration on adhesive contacts at small strains.
\end{center}

\bigskip

\begin{center}
\sc
Luk\'a\v s Adam$^{1,2}$, Ji\v r\'\i\ V. Outrata$^1$, Tom\'{a}\v{s} Roub\'\i\v{c}ek$^{2,3,1}$.
\end{center}

{\footnotesize
\noindent
$^1$ 
Institute of Information Theory and Automation, Academy of Sciences,
Pod vod\'{a}renskou v\v{e}\v{z}\'{\i}~4,
\\\hspace*{1em}CZ--18208~Praha~8, Czech Republic.
\\
$^2$ Faculty of Mathematics and Physics, Charles University, Sokolovsk\'a 83,
CZ--18675~Praha~8, Czech Republic
\\
$^3$ Institute of Thermomechanics of the ASCR, 
Dolej\v skova 5, CZ--18200~Praha~8, Czech Republic.
}

\bigskip

\begin{center}
\begin{minipage}{15.5cm}
{\small
\noindent{\textbf{Abstract}}: A class of evolution quasistatic systems which leads, after a suitable time discretization, to recursive nonlinear programs, is considered and optimal control or identification problems governed by such systems are investigated. The resulting problem is an evolutionary Mathematical Programs with Equilibrium Constraints (MPEC). A subgradient information of the (in general nonsmooth) composite objective function is evaluated and the problem is solved by the Implicit programming approach. The abstract theory is illustrated on an identification problem governed by delamination of a unilateral adhesive contact of elastic bodies discretized by finite-element method in space and a semi-implicit formula in time.
Being motivated by practical tasks, an identification problem of the fracture toughness and of the elasticity moduli of the adhesive is computationally implemented and tested numerically on a two-dimensional example. Other applications including frictional contacts or bulk damage, plasticity, or phase transformations are outlined.

\medskip

\noindent{\textbf{Keywords}}: rate-independent systems, 
optimal control, identification, fractional-step time discretization, 
quadratic programming, gradient evaluation, variational analysis, 
implicit programming approach, limiting subdifferential, coderivative, 
nonsmooth contact mechanics, delamination.

\medskip

\noindent{\textbf{AMS Classification}}: 
35Q90, 
49N10, 
65K15, 
65N38, 
65M32, 
74M15,  
74P10, 
74R20, 
90C20. 
}
\end{minipage}
\end{center}

\section{Introduction}

Many evolution systems have the structure of the {\it generalized gradient 
flow} $$\DT q\in\partial_\xi\Rcal^*(q;-\partial_q\Ecal(t,q))$$ with 
functionals $\Ecal(t,q)$ and $\Rcal^*(q;\xi)$. 
Here $q$ is an abstract state of the system and $\DT q$ denotes its time 
derivative. Quite typically,  $\Rcal^*(q;\cdot)$ is convex and, making 
the conjugate of $\Rcal^*(q;\xi)$ with respect to the ``driving force'' 
variable $\xi$, i.e.\ 
$\Rcal(q;v)=\sup_{\xi}[\langle v,\xi\rangle-\Rcal^*(q;\xi)]$,
the generalized gradient flow can equivalently be written in the 
Biot-equation form 
\begin{align}\label{Biot}
\partial_{\DT q}\Rcal(q;\DT q)+\partial_q\Ecal(t,q)\ni0.
\end{align}
In many cases, the problem is nonsmooth due to 
a nonsmoothness of $\Rcal(q;\cdot)$ or $\Ecal(t,q)$, which is why we wrote
inclusion in \eqref{Biot} rather than equality. Ansatz \eqref{Biot} is
very general and covers great variety of problems in particular in nonsmooth 
continuum mechanics.  The state variable $q$ may involve displacements and various 
internal parameters (but also various concentrations of some constituents 
subjected to diffusive processes).  In this paper, we focus on a 
subclass of such problems where 
the state has the structure 
\begin{align}
q=(u,z)
\end{align} 
for each time instance $t$ in a Banach space $U\times Z$. In this way, a 
quasistatic {\it plasticity}, or {\it damage}, or various 
{\it phase transformations} at small 
strains can be modelled, and also various problems in {\it contact mechanics} 
like friction or adhesion, together with various combinations of these 
phenomena. 

After a suitable time discretization, \eqref{Biot} gives rise to recursive
optimization problems. Often (or, in applications in continuum mechanics
we have in mind, rather typically) $q$ ranges over an infinite-dimensional 
Banach space and, after a possible ``spatial'' discretization,
these minimization problems have a structure of strictly convex 
{\it Quadratic Programming} problems. It is then relatively easy to use 
such a discretized evolution problem as a governing system for some 
optimization problem, e.g., optimal control or identification of parameters. This leads to {\it Mathematical Programs with Evolution Equilibrium Constraints} 
(MPEEC) which have been studied e.g. in \cite{kocvara.kruzik.outrata.2006,kocvara.outrata.2004,outrata.2000}.

The functionals in \eqref{Biot} depend also on an abstract 
parameter $\P$ and have a special form
$$
\aligned
\Rcal(\P,q;\DT q)&=\Rcal_1(\P,u,z;\DT u) +\Rcal_2(\P,u,z;\DT z),\\
\Ecal(t,\pi,q)&=\Ecal(t,\P,u,z).
\endaligned
$$
We then consider an optimal-control 
or an identification problem on a fixed time interval $[0,T]$:
\begin{align}\label{identification}
\left.\begin{array}{ll}
\text{Minimize}&\displaystyle{\int_0^T\!\!
j(u,z)
\dt+
H(\P)}
\\[.2em]
\text{subject to}&\partial_{\DT u}\Rcal_1(\P,u,z;\DT u)+
\partial_u\Ecal(t,\P,u,z)\ni0
\ \ \text{ for a.a.\ }t\!\in\![0,T],
\ \ u(0)=u_0,\ 
\\[.5em]&\partial_{\DT z}\Rcal_2(\P,u,z;\DT z)+\partial_z\Ecal(t,\P,u,z)\ni0
\ \ \text{ for a.a.\ }t\!\in\![0,T],\ \ z(0)=z_0,
\\[.5em]
&u\in L^\infty(0,T;U),\ \ z\in L^\infty(0,T;Z),\ \ \P\in\Pi
\end{array}\right\}
\end{align}
with some $j:U\times Z\to\R$ and $H:\Pi\to\R$ specified later; here
$\Pi$ is a closed convex set of a Banach space where $\P$ lives. 
In some models, 
the flow rule for $u$ in \eqref{identification} is purely static,
i.e.\ $\Rcal_1=0$. In this case, 
if there is no dissipation in this part, then only $z_0$ but not $u_0$ is decisive when considering
$\partial_u\Ecal(t,\P,u_0,z_0)\ni0$. We use the standard notation for Bochner space $L^\infty(0,T;\cdot)$
of Banach-space-valued Bochner measurable functions on $[0,T]$.

The main aim of the paper is a deep analysis of a discretized version of MPEEC \eqref{identification} leading both to sharp necessary optimality conditions as well as to an efficient numerical procedure based on the so-called {\it Implicit Programming} approach (ImP), cf. \cite{luo.pang.ralph.1996,kocvara.outrata.zowe.1998}. In particular, on  the basis of the subdifferential calculus of B. Mordukhovich \cite{mordukhovich.2006,rockafellar.wets.1998} we will show that the \textit{solution map} $ S: \P \mapsto (u,z) $ defined via the discretized equilibrium relations in \eqref{identification}, is single-valued and locally Lipschitz  and satisfies henceforth the basic ImP hypothesis. In the respective proof one 
has to deal with a difficult multifunction arising in connection with our \textit{evolving} constraint sets. The application of standard tools of generalized differential calculus provides us in this case only with an upper estimate of the coderivative of the normal cone mapping to the \textit{overall} constraint set, which could be a substantial drawback both in the optimality conditions as well as in the used numerical approach. To overcome this hurdle, we have employed a new result from \cite{adam.cervinka.pistek.2014} which enables us to compute the limiting coderivative of the mentioned normal cone mapping \textit{exactly}.
 
The plan of the paper is the following. In Section~\ref{sect-disc},
we briefly introduce a suitable discretization of the identification 
problem \eqref{identification} that yields a unique response $(u,z)$
of the constraint system of \eqref{identification} for a given $\P$ and allows 
for efficient optimization technique. After introducing some notation and 
notions from variational analysis in Section~\ref{sect-aux}, we formulate 
in Section~\ref{sect-opt-cond} first-order necessary optimality 
condition for the discrete version of MPEEC \eqref{identification} and derive 
a subgradient formula for the composite objective function of the discretized 
problem. In Section~\ref{sect-formulation}, we formulate a specific 
application-motivated identification problem from contact continuum mechanics 
that fits (and illustrates) the system (\ref{identification}). Eventually, 
in Section~\ref{sect-numerics}, we illustrate the usage of the subgradient 
evaluation procedure via an adhesive contact problem in a nontrivial 
two-dimensional example involving a spatial discretization by 
{\it Finite-Element-Method} (FEM).

\section{Discretization of the identification problem \eqref{identification}}
\label{sect-disc}

The natural procedure is to discretize the problem \eqref{identification} in time. This might be a rather delicate problem
especially when the controlled system (\ref{identification}b,c) 
exhibits responses of different time scales, and in particular with
tendencies for jumping, which quite typically happens in rate-independent 
systems
governed by nonconvex potentials $\Ecal(t,\P, \cdot,\cdot)$. 

We consider (for simplicity) an equidistant partition of the time
interval $[0,T]$ with a time step $\tau>0$ such that $T/\tau=:K\in\N$
and then a fractional-step-type semi-implicit time 
discretization of (\ref{identification}b,c).
Moreover, if $U$, $Z$ or $\Pi$ is infinite-dimensional, the time-discrete 
problem still remains infinite-dimensional, and to implement it on computers, we 
need to apply also an  abstract space discretization controlled by a 
parameter, let us denote it by $h>0$. Such an approximation can be considered by replacing $U$, $Z$, and $\Pi$ in \eqref{identification-disc} by their finite-dimensional subsets 
 $U_h$, $Z_h$, and $\Pi_h$. Counting also a possible
numerical approximation of $\Ecal$, denoted by $\Ecal_h$, altogether, \eqref{identification} turns into the problem:
\begin{align}\label{identification-disc}
\left.\begin{array}{ll}
\text{Minimize}\!\!&\tau\,\mbox{\large$\sum$}_{k=1}^{K}
\,j(u_{\tau h}^k,z_{\tau h}^k)+H(\P)
\\[.2em]
\text{subject to}\!\!\!&\displaystyle{
\partial_{\DT u}\Rcal_1\Big(\pi,u_{\tau h}^{k-1},z_{\tau h}^{k-1};\frac{u_{\tau h}^k{-}u_{\tau h}^{k-1}}{\tau}\Big)+
\partial_u\Ecal_h(k\tau,\pi,u_{\tau h}^k,z_{\tau h}^{k-1})\ni0,\ \ u_{\tau h}^0=u_{0h},\!\!}
\\[.5em]&\displaystyle{
\partial_{\DT z}\Rcal_2\Big(\pi,u_{\tau h}^{k-1},z_{\tau h}^{k-1};
\frac{z_{\tau h}^k{-}z_{\tau h}^{k-1}}{\tau}\Big)
+\partial_z\Ecal_h(k\tau,\pi,u_{\tau h}^k,z_{\tau h}^k)\ni0,\ \ \ \ \,z_{\tau h}^0=z_{0h},
\!\!}
\\[.5em]&
u_{\tau h}^k\in U_h\ \text{ and }\ z_{\tau h}^k\in Z_h
\ \text{ for }k=1,...,K,\ \ \P\in\Pi_h,
\end{array}\right\}\!\end{align}
where 
$(u_{0h},z_{0h})\in U_h{\times}Z_h$ is an approximation of the
initial condition $(u_0,z_0)$. Let us note that the controlled system 
(\ref{identification-disc}b,c)
decouples so that, for a given $\P$, one is to solve alternating optimization 
problems
\begin{subequations}\label{2-minimizations}\begin{align}
&\text{Minimize}\!\!&&
\tau\Rcal_1\Big(\pi,u_{\tau h}^{k-1},z_{\tau h}^{k-1};\frac{u{-}u_{\tau h}^{k-1}}{\tau}\Big)
+\Ecal_h(k\tau,\pi,u,z_{\tau h}^{k-1})&&\!\!\text{subject to }u\!\in\!U_h
\intertext{and, taking (one of) its solution for $u_{\tau h}^k$, further}
 &\text{Minimize}\!\!&&
\tau\Rcal_2\Big(\pi,u_{\tau h}^{k-1},z_{\tau h}^{k-1};\frac{z{-}z_{\tau h}^{k-1}}{\tau}\Big)
+\Ecal_h(k\tau,\pi,u_{\tau h}^k,z)&&\!\!\text{subject to }z\!\in\!Z_h
\end{align}\end{subequations}
which yields $z_{\tau h}^k$ as (one of) its solution. Assuming 
$\Ecal(t,\P,\cdot,\cdot)$ as well as its 
approximation $\Ecal_h(t,\P,\cdot,\cdot)$
separately strictly convex (and, of course, coercive 
with compact level sets) and $\Rcal_i(\P,u,z;\cdot)$ convex, $i=1,2$,
both problems in \eqref{2-minimizations} have unique solutions
 $u_{\tau h}^k$ and $z_{\tau h}^k$, respectively,
and thus the whole recursive problem in the constraint system (\ref{identification-disc}) has 
a unique response for a given $\P$ as well. This allows us to reformulate 
\eqref{identification-disc} as a minimization problem for a functional 
depending on $\P$ only, cf.\ \eqref{Label optimal control discretized general} 
below. This will be exactly the situation we will consider in the rest of this 
article. The fully discretized system \eqref{identification-disc} can thus be 
understood as an MPEC for which a developed theory exists.

In what follows, we will confine ourselves to problems with a bit more detailed
(but nevertheless still fairly general) structure, namely
\begin{subequations}\label{structure}\begin{align}
&\Ecal(t,\pi,u,z)=
\begin{cases}\mathcal{E}(t,\pi,u,z)&\text{if }u\in\Lambda_0^t,\
z\in K_0^t,\\
\infty&\text{otherwise},\end{cases}
\\
&\Rcal_1(\pi,u,z,\DT u)=\mathcal{R}_1(\pi,u,z,\DT u),
\\
&\Rcal_2(\pi,u,z,\DT z)=\begin{cases}\mathcal{R}_2(\pi,u,z,\DT z)
&\text{if }\DT z\in K_1,\\
\infty&\text{otherwise},\end{cases}
\end{align}\end{subequations}
where $\mathcal{E}$, $\mathcal{R}_1$, and $\mathcal{R}_2$ are finite
and smooth, $\Lambda_0^t$, $K_0^t$, and $K_1$ are convex closed set, the last 
one being a cone. We will use $\mathcal{E}_h$ as a possible approximation 
of $\mathcal{E}$.

Although, in Section~\ref{sect-numerics}, we will illustrate usage of this 
model on a rather special inverse adhesive-contact problem, most of the 
considerations can expectedly be applied (after possible modification) to 
many other problems from continuum mechanics and physics, as (various 
combination of) damage, phase-transformations, plasticity, etc.

\begin{remark}[Stability and convergence for $\tau\to0$ and $h\to0$]
\label{rem-convergence}
The focus of this article is on the identification of the discrete 
finite-dimensional problem. Nevertheless, the convergence towards the 
original continuous problem when $\tau\to0$ and $h\to0$ is of interest.

Without going into (usually rather technical) details, let us only mention 
that under certain qualification of $\Rcal_1$, $\Rcal_2$ and $\Ecal_h$,
a boundedness (=\,numerical stability) and convergence
of a solution $(u_{\tau h},z_{\tau h})$ to the discrete state problem  
obtained by interpolation from values $(u_{\tau h}^k,z_{\tau h}^k)_{k=1}^{K}$ 
towards a weak solution $(u,z)$ to controlled state system for a fixed $\P$ can usually be shown at least in terms of subsequences
in various situations. A rather simple situations is if 
$\Rcal_2$, or possibly also $\Rcal_1$, is uniformly convex; this
corresponds to some viscosity.
In a special fully rate-independent case when $\Rcal_1=0$ and 
$\Rcal_2(\pi,u,z;\cdot)$ is 1-homogeneous and independent of $(u,z)$, 
such convergence was proved in \cite{roubiceka}; in this case the 
weak solutions are called local solutions. The
uniqueness is however not guaranteed in general. If 
$\Ecal(t,\P,\cdot,\cdot)$ is jointly uniformly convex, then
this uniqueness and even continuous dependence on $\P$ holds,
cf.\ \cite{mielke.2006,MieRou??RIS} for a survey of such situations. 
This is e.g.\ the case of linearized rate-independent plasticity with 
hardening. Sometimes, viscosity can help for this uniqueness.
This is the case of 
frictional normal-compliance contact of visco-elastic bodies
which, after a certain algebraic manipulation gets the 
structure with $\Ecal(t,\P,\cdot,\cdot)$ separately uniformly quadratic
with linear constraints in two-dimensions, cf.\ \cite{roubicek.2013},
or with conical constraints in three-dimensions.
The uniqueness of the response of the continuous problems
was shown in \cite{HaShSo01VNAQ}.

As usual, the convergence of solutions to \eqref{identification-disc}
towards solutions to \eqref{identification} is much more delicate
and it is a well-known fact that it cannot be expected 
unless the controlled state system in \eqref{identification} has a unique 
response or at least any solution to \eqref{Biot} can be 
attained by the discretized solutions, which is however usually not
granted unless the solution to \eqref{Biot} is unique. In any case, one needs 
to show the continuous convergence of the solution map 
$S_{\tau h}:\P\mapsto(u_{\tau h},z_{\tau h})$,
i.e. that $\tau\to0$ and $h\to0$ and $\WT \P\to\P$ implies 
$S_{\tau h}(\WT \P)\to S(\P)$. This is usually a relatively simple modification of the convergence for $\P$ fixed.
\end{remark}

\section{Notation and selected notions of variational analysis}
\label{sect-aux}
Having in mind the discrete problem with $\tau>0$ and $h>0$, we will 
use notation
\begin{subequations}\label{Label definition P Q}\begin{align}
&p_{\tau h}^k(\P,\WT  u,u,\WT  z) := \nabla_u
\mathcal{R}_1\Big(\P,\WT  u,\WT  z,\frac{u-\WT  u}\tau\Big)
+\nabla_u\mathcal{E}_h\big(k\tau,\P,u,\WT  z\big),\\&
q_{\tau h}^k(\P,\WT  u,u,\WT  z,z) := \nabla_z 
\mathcal{R}_2\Big(\P,\WT  u,\WT  z,\frac{z-\WT  z}\tau\Big)
+\nabla_z\mathcal{E}_h\big(k\tau,\P,u,z\big),
\\&\KK^k(\tilde z):=\big(K_1+\tilde z\big)\cap K_0^{k\tau},\ \ \text{ and}
\\&
J(\P,\hat u,\hat z):=\tau\,\mbox{\large$\sum$}_{k=1}^{K}
\,j(u^k,z^k)+H(\P)\ \ \text{ with }\ \hat u=(u^1,...,u^K)\ \text{ and }\ 
\hat z=(z^1,...,z^K).
\end{align}\end{subequations}
Since problems \eqref{2-minimizations} are convex, necessary optimality conditions are also sufficient and thus, taking into account structure \eqref{structure}, problem \eqref{identification-disc} can equivalently be written in the form:
\begin{align}\label{Label state equation inclusion original}
\left.
\!\!\begin{array}{ll}
\text{Minimize}\!\!& J(\P,u_{\tau h},z_{\tau h})\ \ \text{ with }\
u_{\tau h}:=(u_{\tau h}^1,...,u_{\tau h}^K)\ \text{ and }\ 
z_{\tau h}=(z_{\tau h}^1,...,z_{\tau h}^K)
\\[.5em]
\text{subject to}\!\!\!&0\in p_{\tau h}^k(\P,u_{\tau h}^{k-1},u_{\tau h}^k,z_{\tau h}^{k-1})
+\Nc_{\Lambda_\tau^k}(u_{\tau h}^k),\ \ \ \ \ \
k=1,\dots,K,\ \ \ \ \ \ u_{\tau h}^0=u_{0h},\!\!
\\[.5em]
&0\in q_{\tau h}^k(\P,u_{\tau h}^{k-1},u_{\tau h}^k,z_{\tau h}^{k-1},z_{\tau h}^k)
+\Nc_{\KK^k(z_{\tau h}^{k-1})}(z_{\tau h}^k),\ k=1,\dots,K,\ z_{\tau h}^0=z_{0h},\!\!\!
\\[.5em]
&\P\in\Pi_h\end{array}\right\}\!
\end{align}
with $K=T/\tau$. Defining the solution map $S_{\tau h}:\P\mapsto (u,z)$ implicitly via system 
(\ref{Label state equation inclusion original}), we may used the so-called 
implicit programming approach to rewrite problem 
\eqref{Label state equation inclusion original} equivalently into the form
\begin{align}\label{Label optimal control discretized general}
\text{Minimize}\ \ J(\P,S_{\tau h}(\P))\ \ 
\text{subject to}\ \ \P\in\Pi_h.
\end{align}
In the rest of the paper we will make use of the following standing 
assumption, which imply in particular the single-valuedness of the solution map $S_{\tau h}$:
\begin{itemize}
\item[(A1):]\ \ 
$\Ecal_h(t,\pi,u,\cdot)$ and $\Ecal_h(t,\pi,\cdot,z)$ are strictly 
convex,
\item[(A2):]\ \ 
$\Rcal_1(\P,u,z,\cdot)$ and $\Rcal_2(\P,u,z,\cdot)$ are convex,
\item[(A3):]\ \  
$p_{\tau h}^k(\P,\ut,\cdot,z)$ and $q_{\tau h}^k(\pi,\ut,u,\zt,\cdot)$ are 
continuously differentiable mappings, and
\item[(A4):]\ \ 
$\Lambda^k$ and $\KK^k(\zt)$ are closed convex sets.
\end{itemize}
Note that (A1)--(A3) implies that $p_{\tau h}^k(\P,\ut,\cdot,z)$ and 
$q_{\tau h}^k(\pi,\ut,u,\zt,\cdot)$ have  a positive definite Jacobian.

In what follows, we will fix time (and, if any, also space) discretization
and thus we will omit $\tau$ and $h$ in the following sections. The dimension of $U_h$, $Z_h$, and $\Pi_h$ will be respectively 
denoted by $N$, $M$, and $L$.

Before devising a (necessarily) quite complicated procedure to evaluate a 
gradient information for the nonsmooth functional $\P\mapsto J(\P,S(\P))$,
let us still briefly present basic notions from variational analysis which are 
essential for this paper. More information can be found 
in \cite{rockafellar.wets.1998} for finite-dimensional setting 
or in \cite{mordukhovich.2006} and \cite{clarke.ledyaev.stern.wolenski.1998} 
for the general infinite-dimensional case. 

All objects in this section are finite-dimensional. For sequence of sets 
$A_k\subset\R^n$ we define the Painlew\'e-Kuratowski upper limit as
$$
\Limsup_{k\to\infty} A_k=\big\{x\,|\ \exists x_k\!\in\!A_k,\ x
\text{ is an accumulation point of }\{x_k\}\big\}.
$$
Using this construction, we define for any $\bar{x}\in A$ the Bouligand tangent cone, Fr\'{e}chet normal cone and the limiting normal cone respectively as
$$
\aligned
\Tc_A(\xb)&=\big\{v\,|\ \exists v_k\to v,\ \lambda_k\searrow 0,\ 
\xb+\lambda_kv_k\in A\big\},
\\
\hat{\Nc}_A(\bar{x})&=(\Tc_A(\xb))^*
=\big\{x^*\,|\ \ssoucin{x^*}{v}\leq 0\text{ for all }v\in T_A(\xb)\big\},\\
\Nc_A(\bar{x})&=\Limsup_{x\stackrel{A}{\to}\bar{x}}\hat{\Nc}_A(x),
\endaligned
$$
where by $x\stackrel{A}{\to}\bar{x}$ we understand $x\to\bar{x}$ with $x\in A$. 
If $\Nc_A(\Pb)=\hat{\Nc}_A(\Pb)$, then we say that $\Pb$ is a regular point of 
$A$, otherwise it is a nonregular point.

To a function $f:\R^n\to \R\cup\{\infty\}$ we can define 
its subdifferential as
$$
\partial f(\bar{x})
=\big\{x^*\,|\ (x^*,-1)\in\Nc_{\epi f}(\bar{x},f(\bar{x}))\big\}.
$$
If $A$ happens to be convex, both normal cones coincide and are equal to the 
normal cone in the standard sense of convex analysis
$$
\Nc_A(\bar{x})=\big\{x^*\,|\ \ssoucin{x^*}{x-\bar{x}}\le0\text{ for all }x\in A\big\}\\
$$
and similarly, if $f$ is continuously differentiable at $\bar{x}$, then $\partial f(\bar{x})=\{\nabla f(\bar{x})\}$.

For a set-valued mapping $M:\R^n\rrarrow\R^m$ and for any $(\xb,\yb)\in \gph M$ we define the coderivative 
$\coder M(\bar{x},\bar{y}):\R^m\rrarrow\R^n$ at this point as
$$
\coder M(\bar{x},\bar{y})(y^*)=\big\{x^*\,|\ (x^*,-y^*)\in\Nc_{\gph M}(\bar{x},\bar{y})\big\}
$$
If $M$ is single-valued, we write only $\coder M(\bar{x})(y^*)$ instead of $\coder M(\bar{x},M(\bar{x}))(y^*)$. If $M$ is single-valued and smooth, then its coderivative amounts to the adjoint Jacobian.

A set-valued mapping $M:\R^n\rrarrow\R^m$ has the Aubin property at $(\bar{x},\bar{y})\in\gph M$ if there exist a constant $L$ and neighborhoods $U$ of $\bar{x}$ and $V$ of $\bar{y}$ such that for all $x, x'\in U$ the following inclusion holds true
$$
M(x)\cap V\subset M(x')+L\norm{x-x'}B(0,1),
$$
where $B(0,1)\subset\R^m$ is the unit ball. If $M$ is single-valued, then the Aubin property reduces exactly to locally Lipschitzian property.

\begin{example}
For a short illustration of the afore-mentioned objects, we use a simple 
example whose extension will be used later in the text. Consider 
$C:=[0,1]\subset\R$. Since $C$ is convex, both normal cones coincide and we have
$$
\gph\Nc_C=\gph\hat{\Nc}_C=\big(\{0\}{\times}\R_-\big)\cup\big([0,1]{\times}\{0\}\big)\cup\big(\{1\}{\times}\R_+\big).
$$
Fix now $(\xb,\yb)=(0,0)$ and compute
$$
\aligned
\hat{\Nc}_{\gph\Nc_C}(\xb,\yb)&=\R_-{\times}\R_+
\ \ \text{ and }\ \ 
\Nc_{\gph\Nc_C}(\xb,\yb)&=\big(\R_-{\times}\R_+\big)\cup\big(\{0\}{\times}\R_+\}\big)\cup\big(\R_-{\times}\{0\}\big).
\endaligned
$$
One can see that in this case limiting normal cone is strictly greater than Fr\'echet one. This means that $(0,0)$ is a nonregular point. Similarly, one can see that so is also $(1,0)$.
\end{example}

\section{Evaluation of a subgradient of $\P\mapsto J(\P,S(\P))$
         and first-order necessary optimality conditions for 
         \eqref{Label state equation inclusion original}}
\label{sect-opt-cond}

To solve problem \eqref{Label state equation inclusion original} or equivalently \eqref{Label optimal control discretized general} efficiently, we need to compute a subgradient information for the mapping 
$\P\mapsto J(\P,S(\P))$. Unfortunately, we cannot expect that $S$ is a differentiable function and 
thus, we need first to compute some kind of generalized derivative of $S$. 

We will work with the generalized differential calculus of Mordukhovich
\cite{mordukhovich.2006, rockafellar.wets.1998} and compute the limiting 
subdifferential of the objective in 
\eqref{Label optimal control discretized general}. To be able to do so, we 
first have to compute the so-called coderivative $\coder S$, which for 
continuously differentiable functions amounts to the adjoint Jacobian. First we state a lemma which links these two concepts together.

\begin{lemma}\label{Lemma noc general form} Consider the solution mapping 
$S:\P\mapsto(\ub,\zb)$ being implicitly defined by system 
\eqref{Label state equation inclusion original} and fix some 
$(\ub,\zb)= S(\Pb)$. Assume that $S$ is Lipschitz continuous on some neighborhood of 
$\Pb$ and that $J$ is continuously differentiable on some neighborhood of 
$(\Pb,\ub,\zb)$. Denoting $\tilde{J}(\P):=J(\P,S(\P))$, we have
$$
\partial \tilde{J}(\Pb) \subset \nabla_{\P}J(\Pb,\ub,\zb)
+\coder S(\Pb,\ub,\zb)(\nabla_uJ(\Pb,\ub,\zb),\nabla_zJ(\Pb,\ub,\zb)).
$$
\end{lemma}
\begin{proof}
It follows directly from \cite[Theorem 7]{mordukhovich.nam.yen.2007} and 
\cite[Exercise 8.8]{rockafellar.wets.1998}.
\end{proof}

To obtain the necessary optimality conditions in the form of original data, we need to compute $\coder S$. This is conducted in the next lemma which will also be the basis for proving the Lipschitzian continuity of $S$ later in Corollary \ref{Corol S Lipschitz}.

\begin{lemma}\label{Lemma coderivative estimate}
Consider the setting of the solution mapping $S:\P\mapsto(\ub,\zb)$ being implicitly defined by system \eqref{Label state equation inclusion original} and fix some $(\ub,\zb)= S(\Pb)$. Assuming (A1)--(A4), the upper estimate of $\coder S(\Pb,\ub,\zb)(u^*,z^*)$ is the collection of all quantities
\begin{equation}\label{Label coderivative upper estimate}
-\sum_{k=1}^K(\nabla_\P p^k)^\top\beta^k - \sum_{k=1}^K(\nabla_\P q^k)^\top\delta^k
\end{equation}
such that for $k=1,\dots,K$ the adjoint equations 
\begin{subequations}\label{Label coderivative adjoint equations}
\begin{align}
\label{Label coderivative adjoint equations 1} -u^{*k}&=\alpha^k - (\nabla_u p^k)^\top\beta^k - (\nabla_u q^k)^\top\delta^k - (\nabla_{\ut} p^{k+1})^\top\beta^{k+1} - (\nabla_{\ut} q^{k+1})^\top\delta^{k+1},\ \text{ and}
\\
\label{Label coderivative adjoint equations 2} -z^{*k}&=\gamma^k - (\nabla_z q^k)^\top\delta^k - (\nabla_{\zt} p^{k+1})^\top\beta^{k+1} - (\nabla_{\zt} q^{k+1})^\top\delta^{k+1}
\end{align}
\end{subequations}
with the terminal conditions $\beta^{K+1}=0$ and $\delta^{K+1}=0$ are fulfilled. For the multipliers $\alpha, \beta, \gamma, \delta$ we have the relations
\begin{subequations}\label{Label coderivative multipliers basic form}\label{Label noc multipliers basic form}
\begin{align}
\label{Label coderivative multipliers basic form 1} \bpm \alpha^k
\\\beta^k\epm&\in\Nc_{\gph\Nc_{\Lambda^k}}(\ub^k,-p^k(\Pb,\ub^{k-1},\ub^k,\zb^{k-1})),
\ \ \ \text{ and}\\\label{Label coderivative multipliers basic form 2}
\bpm \gamma\\\delta\epm&\in\Nc_{\gph Q}(\zb, -q(\Pb,\ub,\zb)),
\end{align}
\end{subequations}
where $\gamma=(\gamma^1,...,\gamma^K)$ and $\delta=(\delta^1,...,\delta^K)$
and where, for $u=(u^1,...,u^K)$ and $z=(z^1,...,z^K)$, we have defined
$$
\begin{array}{rll}
q(\P,u,z)&:=\bpm q^1(\P,u^0,u^1,z^0,z^1)\\\dots\\ q^K(\P,u^{K-1},u^K,z^{K-1},z^K)\epm&\!\!: \R^L\times\R^{KN}\times\R^{KM}\to\R^{KM},\ \ \ \text{ and}
\\
Q(z)&:= \bigtimes_{k=1}^K\Nc_{\KK^k(z^{k-1})}(z^k)&\!\!:\R^{KM}\rrarrow\R^{KM}.
\end{array}
$$
\end{lemma}
\begin{proof}
Similarly to $q$ and $Q$, we define
$$
\begin{array}{rll}
p(\P,u,z)&:=\bpm p^1(\P,u^0,u^1,z^0)\\ \dots\\ p^K(\P,u^{K-1},u^K,z^{K-1})\epm&:\R^L\times\R^{KN}\times\R^{KM}\to\R^{KN},\ \ \ \text{ and}\\
P(u)&:=
\bigtimes_{k=1}^K \Nc_{\Lambda^k}(u^k)&:\R^{KN}\rrarrow\R^{KN}\\
\end{array}
$$

We define the following partially linearized mapping
$$
M(\mu,\nu):=\left\{(\P,u,z)\left|\, \aligned \mu&\in p(\bar{\P},\bar{u},\bar{z})+\nabla_u p(\bar{\P},\bar{u},\bar{z})(u{-}\bar{u})+\nabla_zp(\bar{\P},\bar{u},\bar{z})(z{-}\bar{z})+P(u)\\
\nu&\in q(\bar{\P},\bar{u},\bar{z})+\nabla_uq(\bar{\P},\bar{u},\bar{z})(u{-}\bar{u})+\nabla_zq(\bar{\P},\bar{u},\bar{z})(z{-}\bar{z})+Q(z)\endaligned\right.\right\}
$$
and show that it is single--valued and locally Lipschitz around $(0,0)$. Indeed, the relations defining $M$ read for $k=1,\dots,K$
$$
\aligned
\mu^k&\in p^k(\Pb,\ub^{k-1},\ub^k,\zb^{k-1}) + \nabla_u p^k(\Pb,\ub^{k-1},\ub^k,\zb^{k-1})(u^k-\ub^k) \\ &+ \nabla_{\ut}p^k(\Pb,\ub^{k-1},\ub^k,\zb^{k-1})(u^{k-1}{-}\ub^{k-1})+ \nabla_{\zt} p^k(\Pb,\ub^{k-1},\ub^k,\zb^{k-1})(z^{k-1}{-}\zb^{k-1}) + \Nc_{\Lambda^k}(u^k),\\
\nu^k&\in q^k(\Pb,\ub^{k-1},\ub^k,\zb^{k-1},\zb^k) +
\nabla_uq^k(\Pb,\ub^{k-1},\ub^k,\zb^{k-1}\zb^k)(u^k{-}\ub^k)\\
&+ \nabla_{\ut} q^k(\Pb,\ub^{k-1},\ub^k,\zb^{k-1},\zb^k)(u^{k-1}{-}\ub^{k-1})
+ \nabla_z q^k(\Pb,\ub^{k-1},\ub^k,\zb^{k-1})(z^k{-}\zb^k)\\
&+\nabla_{\zt}q^k(\Pb,\ub^{k-1},\ub^k,\zb^{k-1})(z^{k-1}{-}\zb^{k-1}) + \Nc_{\KK^k(z^{k-1})}(z^k)
\endaligned
$$
with $u^0=\ub^0$ and $z^0=\zb^0$. Since the first inclusion is solved for $u^k$ and the second one for $z^k$, we obtain that $M$ is single-valued due to (A1)--(A4). By virtue of \cite[Corollary 3D.5]{dontchev.rockafellar.2009} we further obtain that $M$ is Lipschitz continuous around $\Pb$, so that the system defining $S$ is strongly regular (in the sense of Robinson \cite{robinson.1980}) at $(0,0,\bar{\P},\ub,\zb)$.

This enables us to use \cite[Proposition 3.2]{outrata.2000} and \cite[Theorem 6.14]{rockafellar.wets.1998} to obtain, with $I$ being the identity matrix, that
$$
\Nc_{\gph S}(\Pb,\ub,\zb)\subset\bpm 0&I&0\\ -\nabla_\P p(\bar{\P},\bar{u},\bar{z})&-\nabla_up(\bar{\P},\bar{u},\bar{z})&-\nabla_zp(\bar{\P},\bar{u},\bar{z})\\ 0&0&I\\ -\nabla_\P q(\bar{\P},\bar{u},\bar{z})&-\nabla_uq(\bar{\P},\bar{u},\bar{z})&-\nabla_zq(\bar{\P},\bar{u},\bar{z})\epm^{\!\!\!\top}
\bpm \alpha\\\beta\\\gamma\\\delta\epm.
$$
with some $\alpha,\beta\in\R^{KN}$ and $\gamma,\delta\in\R^{KM}$ satisfying
$$
\bpm \alpha\\\beta\epm\in\Nc_{\gph P}(\bar{u},-p(\bar{\P},\bar{u},\bar{z}))\ 
\ \ \text{ and }\ \ \ 
\bpm \gamma\\\delta\epm\in\Nc_{\gph Q}(\bar{z},-q(\bar{\P},\bar{u},\bar{z})).
$$
Applying the product rule for normal cones \cite[Proposition 6.41]{rockafellar.wets.1998} we obtain the statement of the lemma.
\end{proof}

If $\Lambda^k$ is a polyhedral set, then $\Nc_{\gph\Lambda^k}(\cdot)$ can be computed due to \cite[Theorem 2]{dontchev.rockafellar.1996} or \cite[Proposition 3.2]{henrion.romisch.2007}. For the computation of $\Nc_{\gph Q}(\cdot)$, we will consider two cases of $\KK^k$, specifically
\begin{subequations}\label{Label Omega definition}
\begin{align}
\label{Label Omega definition 1} \KK^k(z^{k-1})&=\R^M\ \ \ \text{ or}\\
\label{Label Omega definition 2} \KK^k(z^{k-1})&=\{z\in\R^M\,|\, 0\leq z\leq z^{k-1}\}
\end{align}
\end{subequations}
where in \eqref{Label Omega definition 2}, the inequality is understood 
componentwise. 
The former case \eqref{Label Omega definition 1} corresponds to 
$K_0^t=K_1=\R^M$, while 
the latter case \eqref{Label Omega definition 2} corresponds to $K_0^t=\R_+^M$ and $K_1=\R_-^M$. The former case is simple because from 
\eqref{Label coderivative multipliers basic form 2} we immediately obtain 
that $\gamma^k=0$ and $\delta^k\in\R^M$.
For the analysis of the more complicated case \eqref{Label Omega definition 2} we recall the 
definition of $Q$ and define its counterpart $\partition$ for a single time 
instant
\begin{align*}
Q(z) &= \{v\in\R^{KM}\,|\ v^k\in\Nc_{[0,z^{k-1}]}(z^k),\ k=1,\dots,K\},\\
\partition(\zt) &:= \{(z,v)\in\R^M\times\R^M\,|\ v\in\Nc_{[0,\zt]}(z)\}.
\end{align*}
This case is more involved than the previous one, because in $ Q $ one has to 
do with normal cones to moving sets whereby (components of) $ z $ arise 
simultaneously both as the arguments as well as parameters specifying the 
movement of the constraint sets. Such a situation occurs typically in 
quasivariational inequalities and has been studied, e.g., in 
\cite{mordukhovich.outrata.2007}. Unfortunately, the results of 
\cite{mordukhovich.outrata.2007} cannot be directly applied here because the 
set
$$
\KK^{k}(0) = \{ 0\}
$$
does not satisfy even the Mangasarian--Fromovitz constraint qualification. 

As we have mentioned in the introduction, it would be possible to use standard 
calculus rules to obtain a formula for $\Nc_{\gph Q}$ based on multiple computations of $\Nc_{\gph \partition}$. 
The graph of $\partition$ can easily be visualized in Figure \ref{Figure gph Qt} and thus, $\Nc_{\gph \partition}$ can be computed by analyzing $8$ parts of $\gph\partition$ separately, see definition of $\partition_i$ below. However, when computing $\Nc_{\gph Q}$ on the basis of $\Nc_{\gph \tilde{Q}}$ and the chain rule from \cite[Theorem 4.1]{henrion.jourani.outrata.2002} one obtains only an upper estimate and not equality.
\begin{figure}[!ht]
\begin{center}
\includegraphics[width=0.7\textwidth]{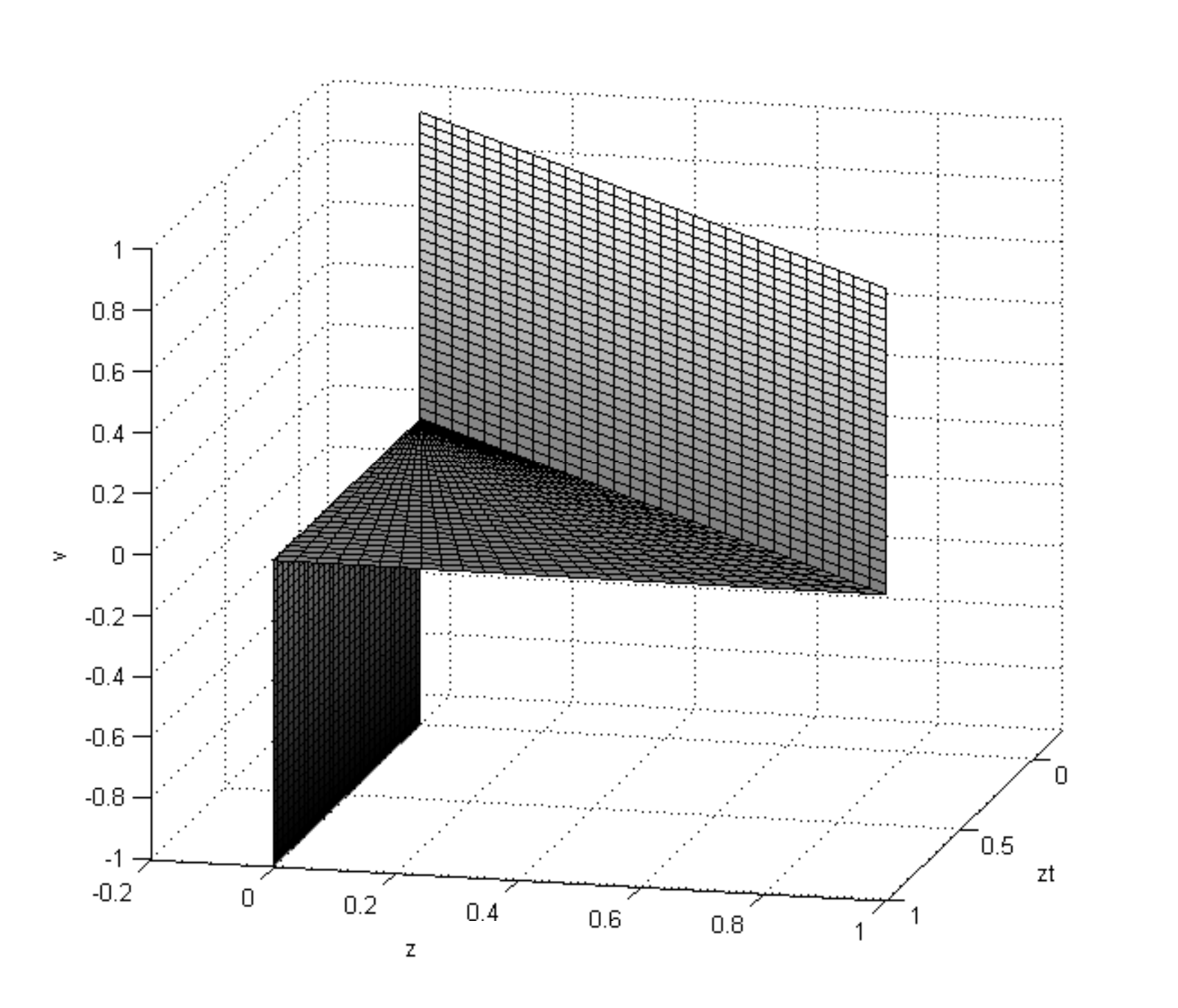}
\end{center}
\vspace*{-2em}
\caption{Visualization of $\gph \partition$.}
\label{Figure gph Qt}
\end{figure}

That is why we make use of \cite{adam.cervinka.pistek.2014} where formulas for Fr\'echet and limiting normal cones to a finite union of convex polyhedra have been derived and then applied to a special structure arising in time dependent problems. For simplicity, we will show the result only for the simplest case of $M=1$. However, the generalization to a more--dimensional space is straightforward and can be conducted in componentwise way.

In the following text we will assume that $i^0=1$. Define now the following sets
\begin{align*}
\partition_1 & = \left \{ (\zt,z,v) \in \R^3 \left | \, \zt \in (0,\infty), z =\zt, v \in (0,\infty) \right . \right \},\\
\partition_2 & = \left \{ (\zt,z,v) \in \R^3 \left | \, \zt \in (0,\infty), z = \zt, v=0 \right . \right \},\\
\partition_3 & = \left \{ (\zt,z,v) \in \R^3 \left | \, \zt \in (0,\infty), z \in (0,\zt), v=0 \right . \right \},\\
\partition_4 & = (0,\infty) \times \{0\} \times \{0\},\\
\partition_5 & = (0,\infty) \times \{0\} \times (-\infty,0),\\
\partition_6 & = \{0\} \times \{ 0 \} \times (-\infty,0),\\
\partition_7 & = \{ 0 \} \times \{0\} \times \{ 0 \},\\
\partition_8 & = \{0\} \times \{0\} \times (0,\infty).
\end{align*}
It is not difficult to show that $\cup_{i=1}^8\partition_i = \gph\partition$. Moreover, $\{\partition_i|\, i=1,\dots,8\}$ forms the so--called normally admissible partition of $\gph\partition$ as defined in \cite{adam.cervinka.pistek.2014}. Now, define the following index sets
$$
\aligned
\Theta &= 
\left\{ (i^1,\dots,i^K) \left| \begin{array}{l}
i^k\in\{1,\dots,8\}\\                                                
i^{k-1}\in\{1,2,3\}\implies i^k\in\{1,2,3,4,5\}\\
i^{k-1}\in\{4,5,6,7,8\}\implies i^k\in\{6,7,8\}\\
\end{array} \right. \right\}\qquad \text{and} \\
I(s) &=
\left\{( i^1,\dots,i^K)\in\Theta \left|\ \begin{array}{ll}
s^k = 1 \implies i^k = 1, & s^k = 5 \implies i^k = 5\\
s^k = 2 \implies i^k\in \{1,2,3\}, & s^k = 6 \implies i^k \in \{5,6\}\\
s^k = 3 \implies i^k = 3, & s^k = 7 \implies i^k \in \{1,\dots,8\}\\
s^k = 4 \implies i^k \in \{3,4,5\}, & s^k = 8 \implies i^k \in \{1,8\}
\end{array} \right.\right\},
\endaligned
$$
where we assume that $s=(s^1,\dots,s^K)\in\{1,\dots,8\}^K$ and all relations are required to hold for all $k=1,\dots,K$. Further define
$$
\partitionbig_{i}:=\left\{ (z,v) \in \R^{2K} \left| \, (z^{k-1},z^k,v^k)\in \partition_{i^k},\ k=1,\dots,K \right. \right \}.
$$
As shown in \cite{adam.cervinka.pistek.2014}, we obtain that $\cup_{i\in\Theta} \partitionbig_i=\gph Q$ and that $\{ \partitionbig_i|\ i\in\Theta\}$ forms a normally admissible partition of $\gph Q$. Now, we may state the result concerning the computation of $\Nc_{\gph Q}(\zb,\vb)$, which replaces the computation of normal cone to a nonconvex sets by the computation of multiple normal cones to convex sets.

\begin{proposition}\cite[Section 4]{adam.cervinka.pistek.2014}\label{Proposition normal cone graph Q}
Fix any $(\zb,\vb)\in\gph Q$ and denote by $\sbar$ the index of the unique component $\partitionbig_{\sbar}$ such that $(\zb,\vb)\in\partitionbig_{\sbar}$. Then
\begin{align*}
\Nc_{\gph Q}(\zb,\vb) = \bigcup_{s\in I(\sbar)} \bigcap_{i\in I(s)} \Nc_{\cl \partitionbig_i}(\partitionbig_s),
\end{align*}
where $\Nc_{\cl\partitionbig_i}(\partitionbig_s)$ denotes the common value $\Nc_{\cl\partitionbig_i}(z,v)$ for any $(z,v) \in \partitionbig_s$. For any $s\in\Theta$ and $i\in I(s)$, this value can be computed as
$$
\Nc_{\cl \partitionbig_i}(\partitionbig_s) = \left \{ \left. \bpm \mu^1+\tilde{\mu}^1\\ \vdots\\ \mu^K+\tilde{\mu}^K \\ \nu^1\\ \dots \\ \nu^K \epm \in \R^{2K} \right |
\aligned \bpm \tilde{\mu}^{k-1}\\ \mu^k \\ \nu^k\epm &\in \Nc_{\cl \partition_{i^k}}(\partition _{s^k}),\ k=1,\dots,K\\ \tilde{\mu}^K &= 0 \endaligned  \right \}.
$$
\end{proposition}

Now we have enough information to prove the Lipschitz continuity of $S$ for both cases in \eqref{Label Omega definition}. Due to assumptions (A1)--(A3) and \cite[Proof of Theorem 2]{dontchev.rockafellar.1996} we obtain that if a pair $(\alpha^k,\beta^k)$ satisfies \eqref{Label coderivative multipliers basic form 1}, then we have $\alpha^{k\top}\beta^k\leq 0$. However, for $(\gamma,\delta)$ satisfying \eqref{Label coderivative multipliers basic form 2}, it may happen that $\gamma^{k\top}\delta^k>0$ (see formula \eqref{Label mu nu property} below). Nevertheless, we are able to overcome this problem by making use of the specific structure of $\gph Q$.

\begin{corol}\label{Corol S Lipschitz}
In the setting of Lemma \ref{Lemma coderivative estimate} assume that $\KK^k$ is defined via \eqref{Label Omega definition 1} or \eqref{Label Omega definition 2}. Fix some $(\ub,\zb)= S(\Pb)$. Then $S$ is Lipschitz continuous around $\Pb$.
\end{corol}
\begin{proof}
Without loss of generality we may assume that $M=1$. Since $S$ is single--valued, it is locally Lipschitz around $\Pb$ if and only if it has the so--called Aubin property around $(\Pb,\ub,\zb)$. Moreover, this property is according to \cite[Theorem 9.40]{rockafellar.wets.1998} equivalent to
\begin{equation}\label{Label S Mordukhovich criterion}
\coder S(\Pb,\ub,\zb)(0,0)=\{0\}.
\end{equation}
To show this, we plug $u^*=z^*=0$ into system \eqref{Label coderivative adjoint equations}--\eqref{Label coderivative multipliers basic form} and deduce that $\beta=\delta=0$, which implies that \eqref{Label coderivative upper estimate} is equal to zero as well, and thus condition \eqref{Label S Mordukhovich criterion} is fulfilled.

To this end, we first realize that the first case \eqref{Label Omega definition 1} implies $\gamma^k=0$ and $\delta^k\in\R^M$. In the rest of the proof, we will consider only the second case \eqref{Label Omega definition 2} with a note that case \eqref{Label Omega definition 1} can be shown by a slight modification of the last paragraph. Fix any $(\gamma,\delta)\in\Nc_{\gph Q}(\zb,-q(\Pb,\ub,\zb))$. From Proposition \ref{Proposition normal cone graph Q} we know that there is some $s\in I(\sbar)$ such that for all $i\in I(s)$ there exist some $\mu_i^k$, $\tilde{\mu}_i^k$ and $\nu_i^k$ such that $\gamma^k=\mu_i^k+\tilde{\mu}_i^k$, $\delta^k=\nu_i^k$, $\tilde{\mu}_i^K = 0$ and relation
\begin{equation}\label{Label mu nu normal cone}
\bpm \tilde{\mu}_i^{k-1}\\ \mu_i^k \\ \nu_i^k\epm \in \Nc_{\cl \partition_{i^k}}(\partition _{s^k})
\end{equation}
holds for all $k=1,\dots,K$.

We will define now the index set
$$
I=\left\{(i^1,\dots,i^K)\left|
\begin{array}{l}
\begin{array}{ll}
s^k=1\implies i^k=1,& s^k=2\implies i^k\in\{1,3\}\\
s^k=3\implies i^k=3,& s^k=4\implies i^k\in\{3,5\}\\
s^k=5\implies i^k=5
\end{array} \\
\begin{array}{ll}
s^k=6,\ i^{k-1}\in\{1,3\}\implies i^k=5\\
s^k=6,\ i^{k-1}\in\{5,6,8\}\implies i^k=6\\
s^k=7,\ i^{k-1}\in\{1,3\}\implies i^k\in\{1,5\}\\
s^k=7,\ i^{k-1}\in\{5,6,8\}\implies i^k\in\{6,8\}\\
s^k=8,\ i^{k-1}\in\{1,3\}\implies i^k=1\\
s^k=8,\ i^{k-1}\in\{5,6,8\}\implies i^k=8
\end{array}
\end{array}
\right. \right\}
$$
and say that property $(P^k)$ holds if
\begin{subequations}\label{Label mu nu nonnegative product}
\begin{align}
\label{Label mu nu nonnegative product 1} s^k\in\{1,2\}&\implies\mu_i^k= 0\text{ for all }i\in I,\text{ and}\\ 
\label{Label mu nu nonnegative product 2} \exists j<k: s^j=4,\ s^{j+1}=\dots=s^k=8 &\implies \mu_i^k\geq 0\text{ for some }i\in I\\ 
\nonumber&\qquad\qquad\text{ with }i^j=3\text{ and }i^{j+1}=\dots=i^k=1.
\end{align}
\end{subequations}
Naturally, this property is satisfied if $s^k\notin\{1,2,8\}$ and it can be shown that $I\subset I(s)$. We will now show that for all $k=1,\dots,K-1$ we have the following implication
\begin{equation}\label{Label mu nu property}
\gamma^{k+1\top}\delta^{k+1}\leq 0\text{ and }(P^{k+1})\text{ holds }\implies \gamma^{k\top}\delta^k\leq 0\text{ and }(P^k)\text{ holds.}
\end{equation}

Thus, we assume $\gamma^{k+1\top}\delta^{k+1}\leq 0$ and that property $(P^{k+1})$ holds. We will now make use of the fact that $\delta^k=\nu_i^k$, and thus $\nu_i^k$ does not depend on $i$. By evaluating \eqref{Label mu nu normal cone}, we obtain that there exists $i\in I\subset I(s)$ such that
$$
\begin{array}{llllllllllll}
s^{k+1}=1\implies \tilde{\mu}_i^k=-\mu_i^{k+1}, &\qquad  s^k=1 \implies &  & \nu_i^k =0,\\
s^{k+1}=2\implies \tilde{\mu}_i^k=-\mu_i^{k+1}, &\qquad  s^k=2 \implies & \mu_i^k\geq 0, & \nu_i^k \leq0,\\
s^{k+1}=3\implies \tilde{\mu}_i^k=0, &\qquad  s^k=3 \implies & \mu_i^k=0, & \\
s^{k+1}=4\implies \tilde{\mu}_i^k=0, &\qquad  s^k=4 \implies & \mu_i^k\leq 0, & \nu_i^k \geq 0,\\
s^{k+1}=5\implies \tilde{\mu}_i^k=0, &\qquad  s^k=5 \implies &  & \nu_i^k =0,\\
&\qquad  s^k=6 \implies &  & \nu_i^k =0,\\
&\qquad  s^k=7 \implies &  & \nu_i^k = 0,\\
&\qquad  s^k=8 \implies & & \nu_i^k =0.
\end{array}
$$
The implication $s^k=7 \implies \nu_i^k = 0$ follows from $I\subset I(s)$, the nondependence of $\nu_i^k$ on $i$ and from the possibility to choose either $i^k\in\{1,5\}$ or $i^k\in\{6,8\}$. We observe now that in any case we have $\mu_i^{k\top}\delta^k = \mu_i^{k\top}\nu_i^k\leq 0$. This means that we have managed to prove $\gamma^{k\top}\delta^k\leq 0$ provided $\tilde{\mu}_i^k=0$ or $\nu_i^k=0$.

Thus, to prove the first part of \eqref{Label mu nu property} it remains to investigate cases $s^{k+1}\in\{1,2,6,7,8\}$ and $s^k\in\{2,3,4\}$. We will restrict now to these problematic cases. If $s^{k+1}\in\{1,2\}$, then $(P^{k+1})$ implies $\tilde{\mu}_i^k = -\mu_i^{k+1}= 0$ and we may apply the previous result. If $s^{k+1}\in\{6,7\}$ and $s^k=4$, then choosing $i^{k+1}=5$ and $i^k=3$ results in $\tilde{\mu}_i^k\leq 0$ and $\mu_i^k\leq 0$, which together with $\nu_i^k\geq 0$ implies $\gamma^{k\top}\delta^k\leq 0$. Due to definition of $\Theta$, it remains to investigate the last case: $s^{k+1}=8$ and $s^k=4$. In this case, we choose $i^{k+1}=1$ and $i^k=3$, which leads to $\mu_i^{k+1}+\tilde{\mu}_i^k\leq 0$ and $\mu_i^k\leq 0$. But since $\mu_i^{k+1}\geq 0$ due to $(P^{k+1})$, we have $\tilde{\mu}_i^k\leq 0$, and thus we again obtain $\gamma^{k\top}\delta^k\leq 0$. So far, we have managed to prove that if the left--hand side of \eqref{Label mu nu property} holds true, then we have $\gamma^{k\top}\delta^k\leq 0$.

To show the validity of formula \eqref{Label mu nu property}, we need to verify that $(P^k)$ holds as well. To do so, we multiply the adjoint equation \eqref{Label coderivative adjoint equations 2} by $\delta^k$, which due to assumption (A1)--(A2) and the already proven $\gamma^{k\top}\delta^k\leq 0$ results in $\gamma^k=\mu_i^k+\tilde{\mu}_i^k=0$ and $\delta^k=\nu_i^k=0$ for all $i\in I$. We will now investigate the cases described on the left--hand side of \eqref{Label mu nu nonnegative product}.

For \eqref{Label mu nu nonnegative product 1} we have $s^k\in\{1,2\}$. This by definition of $\Theta$ yields $s^{k+1}\in\{1,2,3,4,5\}$. If $s^{k+1}\in\{3,4,5\}$, then $\tilde{\mu}_i^k=0$ and thus $\mu_i^k=0$ follows. If on the other hand we have $s^k\in\{1,2\}$, then from assumed $(P^{k+1})$ we get $\tilde{\mu}_i^k=-\mu_i^{k+1}=0$, and thus $\mu_i^k=0$ follows for this case as well. To prove \eqref{Label mu nu nonnegative product 2} consider some $j<k$ and $s^j=4$, $s^{j+1}=\dots=s^k=8$, $i^j=3$ and $i^{j+1}=\dots=i^k=1$. If $s^{k+1}=8$, then $i^{k+1}=1$ and we may apply $(P^{k+1})$ to obtain $\mu_i^{k+1}\geq 0$, which together with $\tilde{\mu}_i^k+\mu_i^{k+1}\leq 0$ and $\mu_i^k+\tilde{\mu}_i^k=0$ implies $\mu_i^k\geq 0$. If $s^{k+1}\in\{6,7\}$, then choosing $i^{k+1}=5$ results in $\tilde{\mu}_i^k\leq 0$, which again implies $\mu_i^k\geq 0$. Since these are all possibilities due to the definition of $\Theta$, we have showed formula \eqref{Label mu nu property}.

Having this formula at hand, the rest of the proof is performed by a finite induction. Since $\tilde{\mu}_i^K=0$, by similar arguments as in the previous text we obtain that $\gamma^{K\top}\delta^K = \mu_i^{K\top}\nu_i^K\leq 0$, which further yields $\gamma^K=\mu_i^K=\delta^K=0$, and thus property $(P^K)$ is satisfied. Hence, we have obtained the validity of the first step for finite induction. Plugging this into the first adjoint equation \eqref{Label coderivative adjoint equations 1} and multiplying it by $\beta^K$, we obtain that $\alpha^K=\beta^K=0$. Since the left--hand side of \eqref{Label mu nu property} is satisfied, we immediately obtain that $\gamma^{K-1\top}\delta^{K-1}\leq 0$ and that $(P^{K-1})$ holds. Performing this procedure $K$ times, we obtain that \eqref{Label S Mordukhovich criterion} indeed holds, which finishes the proof.
\end{proof}

Finally, we summarize the derivation of the necessary optimality conditions in Theorem \ref{Theorem noc} below. Thereby, the normal cone $\Nc_{\gph Q}(\cdot)$ is computed in Proposition \ref{Proposition normal cone graph Q} and for the computation of $\Nc_{\gph\Lambda^k}(\cdot)$ we refer the reader to \cite[Theorem 2]{dontchev.rockafellar.1996} or \cite[Proposition 3.2]{henrion.romisch.2007}. Moreover, when solving system \eqref{Label coderivative multipliers basic form} and \eqref{Label noc adjoint equations}, one may use \cite[Lemma 4.7]{adam.outrata.2014} to its advantage.

\begin{thm}[First-order optimality conditions]\label{Theorem noc}
Consider the setting of the solution mapping $S:\P\mapsto(\ub,\zb)$ implicitly defined by system \eqref{Label state equation inclusion original} and fix some $(\ub,\zb)= S(\Pb)$. Assume (A1)--(A4) and that $J$ is continuously differentiable at $(\Pb,\ub,\zb)$. If $(\Pb,\yb,\zb)$ is a local minimum of problem \eqref{Label state equation inclusion original}, then there exists multipliers $(\alpha,\beta,\gamma,\delta)$ satisfying \eqref{Label coderivative multipliers basic form} such that the optimality condition
\begin{equation}\label{Label noc optimality condition}
0\in \nabla_{\P}J(\Pb,\ub,\zb) -\sum_{k=1}^K(\nabla_\P p^k)^\top\beta^k - \sum_{k=1}^K(\nabla_\P q^k)^\top\delta^k + \Nc_{\Pi}(\Pb),
\end{equation}
the adjoint equations with $k=1,\dots,K$
\begin{subequations}\label{Label noc adjoint equations}
\begin{align}
\label{Label noc adjoint equations 1} -\nabla_{u^k} J(\Pb,\ub,\zb)&=\alpha^k - (\nabla_u p^k)^\top\beta^k - (\nabla_u q^k)^\top\delta^k - (\nabla_{\ut} p^{k+1})^\top\beta^{k+1} - (\nabla_{\ut} q^{k+1})^\top\delta^{k+1},\\
\label{Label noc adjoint equations 2} -\nabla_{z^k} J(\Pb,\ub,\zb)&=\gamma^k - (\nabla_z q^k)^\top\delta^k - (\nabla_{\zt} p^{k+1})^\top\beta^{k+1} - (\nabla_{\zt} q^{k+1})^\top\delta^{k+1}
\end{align}
\end{subequations}
and terminal conditions $\beta^{K+1}=0$ and $\delta^{K+1}=0$ are satisfied.
\end{thm}

\begin{remark}[More general dissipation I]\label{rem-general}
In a number of applications $\Rcal_2$ is finite but nonsmooth at 0 and 
$K_1=Z$. In this case, in the generalized equation system defining $S$, 
one has generally a sum of multifunctions which is typically very 
difficult to handle, cf.\ \cite[Theorem 10.41]{rockafellar.wets.1998}. 
Sometimes, however, an analytic formula for the behavior of $S$ at the single 
time instances can be obtained and then $D^*S$ can be computed by applying 
the (first-order) generalized differential 
calculus \cite{mordukhovich.2006,rockafellar.wets.1998}.

Another possible approach to this situation is to transform it into the form 
considered here, i.e.\ $\Rcal_2$ smooth and a suitable $K_1$. Let us 
illustrate this on a one-dimensional case $Z=\R$ with 
$\Rcal_2(\DT z)=a\max(0,\DT z)+b\max(0,-\DT z)$ with some $a,b\ge0$ and, 
e.g., $\Ecal(z)=\frac12z^2$. Considering artificial variable $(z_1,z_2)$ such 
that $z_1+z_2=z$, we may write
\begin{align}
\Ecal(z_1,z_2)=\frac12(z_1{+}z_2)^2\ \ \text{ and }\ \ \Rcal_2(\DT z_1,\DT z_2)
=\begin{cases}a\DT z_1-b\DT z_2&\text{if }\DT z_1\ge0\text{ and }\DT z_2\le0,\\
\infty&\text{otherwise}.\end{cases}
\end{align}
Such a transformation allows to widen the application range towards 
e.g.\ damage or delamination problems with healing in arbitrary space 
dimension. Another application can be frictional contact \cite{VoMaRo??} or 
adhesive contact with an interfacial plasticity \cite{RoPaMa??LSQR} allowing 
to distinguish less dissipative mode I (opening) from more dissipative mode 
II (shear) in two-dimensional cases. Another, rather academical, application is 
the bulk plasticity with kinematic hardening in one dimension. Naturally, all 
these applications are considered with a suitable space discretization.
\end{remark}

\begin{remark}[More general dissipation II]\label{rem-general-II}
In some applications the cone $K_1$ could be the 2nd-order (Lorentz, or 
colloquially also called ``ice-cream'') cone, defined in $\R^l$ by
$$
\{x \in \mathbb{R}^{l} \,|\ x_{l}\ge|(x_1,...,x_{l-1}) |\}.
$$
where $| \cdot |$ stands for the Euclidean norm. In this case it is possible 
to make use of coderivatives of the normal cone mapping associated with 
second-order cones which have been computed in \cite{outrata.sun.2008}. Then, 
however, the special technique of \cite{adam.cervinka.pistek.2014}, tailored 
to polyhedral multifunctions, cannot be used any more and we have to confine 
ourselves to standard calculus rules, which leads to less selective necessary 
optimality conditions.

Typical applications of this type with $K_1=Z$ are a frictional contact in 
three-dimensional case or plasticity with kinematic hardening in two- or 
three-dimensional case, again having in mind a suitable space discretization 
in each case. An example which uses a combination of $K_1\ne Z$ with a 
nonsmooth potential $\mathscr{R}_2$, both being of the ``ice-cream-type'', is 
plasticity with isotropic hardening, cf.\ 
\cite{HanRed99PMTN,mielke.2006,Stef08VPHE} which has the dissipation potential 
acting on the rate of $z=(p,\eta)$ of the form:
\begin{align}\label{plast}
\delta_S^*(\DT p)+\delta^{}_{K_1}(\DT p,\DT\eta)
\ \ \ \text{ with}\ \ 0\in S\subset\R^{d\times d}_\mathrm{dev}\ \text{ and }\ 
K_1:=\big\{(\DT p,\DT\eta)\in\R^{d\times d}_\mathrm{dev}
{\times}\R;\ \DT\eta \geq q_{_\mathrm{H}}\delta_S^*(\DT p)\big\}
\end{align}
where $q_{_\mathrm{H}}>0$ and 
$\R^{d\times d}_\mathrm{dev}:=\{A\in\R^{d\times d};\ A=A^\top,\ \mathrm{tr}\,A=0\}$,
and $\delta_A$ stands for an indicator function of a convex set $A$ and 
$\delta_A^*$ of its conjugate. Typically, $S$ a ball, which makes 
both $\delta_S^*$ and $K_1$ of the  ``ice-cream-type''.

A combination of the preceding case with a general polyhedral convex set $K_0$ is also 
possible. This combination allows for some applications in identification of 
parameters of some phenomenological models of phase transformations in
certain ferroic materials as shape-memory alloys where $K_0$ forms 
constraints on an internal variable like $p$ in \eqref{plast} and may 
be considered polyhedral, cf.\ the polycrystalic models in \cite{FrBeSe14MMCM,SFBBS12TMNT}, possibly also in combination with 
plasticity like that one in \eqref{plast}, cf.\ \cite{AuReSt07TDMD,SadBha07MICM}.
\end{remark}

\section{Adhesive contact problem and its identification}
\label{sect-formulation}

We illustrate the above abstract identification problem \eqref{identification} 
on an unilateral adhesive-contact problem for a linear elastic body at 
small strains. We consider $\Omega\subset\R^2$ a Lipschitz domain with 
$\GC\subset\partial\Omega$
and $\GD\subset\partial\Omega$ disjoint parts of the boundary 
$\partial\Omega$ where the delamination is undergoing and time-varying 
Dirichlet boundary condition where $\wD(t)$ is prescribed, respectively.
Now, $u:\Omega\to\R^2$ is the displacement and $z:\GC\to[0,1]$ is 
a delamination parameter having the meaning of the portion of bonds 
of the adhesive which are not debonded. With $\C$ the tensor of elastic moduli,
$h:[0,1]\to\R$ a convex adhesive-stored-energy function,
and with $e(u)$ denoting the small-strain
tensor, i.e.\ $[e(u)]_{ij}=\frac12\frac{\partial u_i}{\partial x_j}
+\frac12\frac{\partial u_j}{\partial x_i}$, we will consider the boundary-value 
problem
\begin{subequations}\label{class-form}\begin{align}\label{class-form-a}
&\mathrm{div}\,\C e(u)=0&&\text{in }[0,T]\times\Omega,
\\\label{class-form-b}&\C e(u)\vec{n}=0&&
\text{on }[0,T]\times(\Gamma{\setminus}(\GC\cup\GD)),
\\&u|_{\GD}=\wD(t,\cdot)&&\text{on }[0,T]\times\GD,
\\\label{class-form-d}
   &\left.\begin{array}{ll}
&\hspace{-1.7em}\uN\ge0,
\quad
z\kappaN\uN+\vec{n}^\top\C e(u)\vec{n}\ge0,
\\[.3em]
   &\hspace{-1.7em}
\big(z\kappaN\uN+\vec{n}^\top\C e(u)\vec{n}\big)\uN=0
\!\!\!\!\!
\\[.3em]
   &\hspace{-1.7em}\DT{z}\le0,\ \ \ \ \xi+\alphaF\ge0,\ \ \ \ \DT{z}(\xi{+}\alphaF)=0,
 \\[.3em] &\hspace{-1.7em}
\xi\,+\,h'(z)\,\,+\,
\frac12(\kappaN\uN^2{+}\kappaT\uT^2\big)\,-\,\,\eps\divS\nablaS z\,\ge\,0,\ \  z\ge0,\!\!\!
 \\[.3em] &\hspace{-1.7em}
\big(\xi+h'(z)+\frac12(\kappaN\uN^2{+}\kappaT\uT^2\big)-\eps\divS\nablaS z\big)\,z=0,
\\[.3em]
   &\hspace{-1.7em}
z\kappaT\uT+\C e(u)\vec{n}-(\vec{n}^\top\C e(u)\vec{n})\vec{n}=0,
   \end{array}\right\}\!\!\!\!
   &&\text{on }[0,T]\times\GC,
\end{align}
\end{subequations}
where we used the decomposition of the trace of displacement 
$u=\uN\vec{n}+\uT$ with $\uN$ being the normal displacement defined as 
$u\cdot\vec{n}$ and $\uT$ being the tangential displacement on $\GC$, and where
$\nablaS$ denotes a ``surface gradient'', i.e.\ the tangential 
derivative defined as $\nablaS z=\nabla z-(\nabla z{\cdot}\vec{n})\vec{n}$ 
for $z$ defined around $\GC$. Alternatively, pursuing the 
concept of fields defined exclusively on $\GC$, we can consider 
$z:\GC\to\R$ and extend it to a neighborhood of $\GC$ and
then again define $\nablaS z:=(\nabla z)P$ 
with $P={\mathbb I}-\vec{n}\otimes\vec{n}$ onto a tangent space, which, in 
fact, does not depend on the particular extension.
Moreover, $\divS:=\mathrm{tr}\,\nablaS$. Then $\divS\nablaS$ is the 
so-called Laplace-Beltrami operator.

Let us remark that \eqref{class-form-a} is the force equilibrium, 
 \eqref{class-form-b} prescribes the zero-traction (i.e.\ free
surface) on $\Gamma{\setminus}(\GC\cup\GD)$. 
The condition \eqref{class-form-d} combines three complementarity
problems related respectively to the Signorini unilateral contact for the displacement $u$, the non-negativity constraint for $z$, 
and the unidirectionality constraint (i.e.\ the non-positivity constraint
on $\DT z$), and eventually the equilibrium of tangential stress.
More in detail, the last two mentioned complementarity problems
write in the classical formulation as the inclusion 
$\partial\delta_{[-\alphaF,\infty)}^*(\DT z)\ni\xi$
with the admissible driving force fulfilling the inclusion 
$\xi\in-\partial_z\Ecal(t,\pi,u,z)
=\eps\divS\nablaS z-h'(z)-
\frac12(\kappaN\uN^2{+}\kappaT\uT^2)-\Nc_{[0,\infty)}(z)$.

Referring to the abstract problem \eqref{Biot}, the 
boundary-value problem \eqref{class-form} corresponds to the 
stored and the dissipation energies
\begin{subequations}\label{ansatz+}\begin{align}
&\Ecal(t,\P,u,z):=\begin{cases}
\displaystyle{\int_{\GC}\frac{1}{2}z
\big(\kappaN\uN^2+\kappaT\uT^2\big)
+
h(z)+\frac{1}{2}\eps\nablaS z{\cdot}\nablaS z
\dS}\hspace*{-9em}&
\\[-.3em]\qquad\quad
+\displaystyle{\int_\Omega\frac{1}{2}\C e(u)\colon e(u)\dx}
&\text{if }u|_{\GD}=\wD(t,\cdot)\text{ on $\GD$ and }
\\[-.5em]&u|_{\GC}{\cdot}\vec{n}\ge0\text{ and }z\ge0\text{ on $\GC$},
\\[-.1em]\qquad\infty&\text{otherwise},
\end{cases}\label{ansatz+E}
\\&\label{ansatz+R}
\Rcal_1\equiv0,\ \ \ \ \Rcal_2(\DT z)
:=\!\begin{cases}
\displaystyle{\int_{\GC}\!\!\alphaF|\DT z|\dS}\!\!\!\!\!
&\text{ if }\DT z\le0\text{ a.e.\ on }\GC,
\\[-.3em]\qquad\infty&\text{ otherwise},
\end{cases}\quad\text{ with }\P=(\alphaF,\kappaN,\kappaT),
\end{align}

Note that $\Ecal(t,\P, \cdot,\cdot)$ is not convex 
but it is separately convex and, if $\GD$ is non-empty and $h$ is strictly
convex, it is  separately strictly convex, complying with our assumption 
(A1)--(A2). Considering $h$ quadratic, this leads, after a suitable 
spatial discretization of \eqref{identification-disc}, to recursive 
alternating strictly convex 
{\it Quadratic-Programming} (QP) which can be solved by efficient
prefabricated software packages. 

The (distributed) parameters to be identified will be the fracture 
toughness $\alphaF$ and the elasticity-moduli of the adhesive $\kappaN$ and 
$\kappaT$, i.e.\ we have considered simply $\P=(\alphaF,\kappaN,\kappaT)$ as
outlined in \eqref{ansatz+R}.  
This choice has a certain motivation in engineering where, in contrast to 
essentially all the bulk material properties, these parameters are largely
unknown and have to be set up in a rather ad-hoc way to fit at least roughly some experiments,
cf.\ e.g.\ \cite{TMGCP10ACTA,TMGP11BACO} based on experiments from 
\cite{JCMO07ENEE}. Actually, the models of adhesive contacts used in 
engineering may be more complicated; typically they distinguish modes of 
delamination (opening vs shear) and/or may involve friction.
Identification of 
friction/adhesive contacts may have interesting applications in geophysics 
where such contact surfaces (called faults) are deep in lithosphere 
and not accessible to direct investigations although a lot of indirect
data from earthquakes are usually available; a popular rate-and-state 
friction model involves one internal parameter (called ageing) which is 
analogous to the delamination parameter used here, cf.\ \cite{Diet07ARSD} for 
a survey or also e.g.\ \cite{Roub14NRSD}. Other models that may
lead to a recursive QP have been mentioned in Remark~\ref{rem-general},
in contrast to problems from  Remark~\ref{rem-general-II} that would lead
to a recursive Second-Order Cone Programming (SOCP) for which
efficient codes do exist, cf.\ \cite{AliGol03SOCP}.

We prescribe some initial conditions $u_0\in H^1(\Omega)$ and 
$z_0\in H^1(\GC)$, $0\le z_0\le1$; note that then $0\le z\le1$ is 
satisfied during the whole evolution process. We further consider a fixed time
horizon $T>0$ and
assume that we have some given desired response 
$(u^{}_\mathrm{d},z^{}_\mathrm{d})$ corresponding e.g. to some experimentally 
obtained measurements, and we want to identify parameters $\P$ such that the 
response $(u,z)=S(\P)$ is as close to $(u^{}_\mathrm{d},z^{}_\mathrm{d})$ as 
possible, i.e.\ we want to minimize the objective 
\begin{align}
\int_0^T\Big[\int_\Omega\frac\zeta2\big|u-u_\mathrm{d}\big|^2\dx
+\int_{\GC}
\frac12\big|z-z_\mathrm{d}\big|^2\dS\Big]\dt
\end{align}
\end{subequations}
where 
$\zeta$ is a fixed weight balancing both parts of the objective function.

After the semi-implicit time discretization, the whole problem \eqref{ansatz+}
reads as
\begin{subequations}\label{Label optimal control problem cont}
\begin{align}\label{Label optimal control problem cont upper}
&\left.
\begin{array}{ll}
\text{Minimize}&\displaystyle{
\tau\sum_{k=1}^K
\Big[\int_\Omega\frac\zeta2\big|u^k-u_\mathrm{d}^k\big|^2\dx
+\int_{\GC}\frac12
\big|z^k-z_\mathrm{d}^k\big|^2\dS\Big]}\\[1em]
\text{subject to}&(u^k,z^k)= S^k(\P,u^{k-1},z^{k-1}),
\ \ \ k=1,...,K,\ \ \text{ and }\\
&\P=(\alphaF,\kappaN,\kappaT)\in\Pi,
\end{array}
\right\}
\intertext{
where the solution map $S^k:(\P,u^{k-1},z^{k-1})\mapsto (u^k,z^k)$ for a particular time instant is now defined by the alternating recursive system: given $\P=(\alphaF,\kappaN,\kappaT)$ and previous values $(u^{k-1},z^{k-1})$, the first one is 
solved for $u^k$ and the second one for $z^k$ recursively for $k=1,...,K$:}
\label{Label state equation optimization cont lower 1}
&\left.
\begin{array}{ll}
\underset{u\in H^1(\Omega,\R^d)}{\text{Minimize}}\ 
&\displaystyle{\int_\Omega\frac{1}{2}\C e(u)\colon e(u)\dx
+\int_{\GC}\frac{1}{2}z^{k-1}
\big(\kappaN\uN^2+\kappaT\uT^2\big)
\dS}\\[1em]
\text{subject to}\ &u|_{\GD}=\wD^k:=\wD(k\tau,\cdot)\ \ \text{ and }\ \ 
u|_{\GC}{\cdot}\vec{n}\geq 0,
\end{array}\hspace*{4.7em}\right\}
\\
\label{Label state equation optimization cont lower 2}
&\left.
\begin{array}{ll}
\!\!\!\!\!\!\underset{z\in H^1(\GC)\cap L^\infty(\GC)}{\text{Minimize}}\ 
&\displaystyle{\int_{\GC}\Big[
h(z) 
{+}\frac{\eps}2\nablaS z{\cdot}\nablaS z
+ \Big(\frac{1}{2}
\big(\kappaN(\uN^k)^2_{}{+}\kappaT(\uT^k)^2_{}\big)
-\alphaF\Big) z\Big]\dS}\\[1em]
\text{subject to}\ &0\leq z\leq z^{k-1}.
\end{array}\hspace*{0em}\right\}
\end{align}
\end{subequations}

Discretizing system \eqref{Label state equation optimization cont lower 1} 
via finite elements, we obtain
\begin{equation}\label{Label state equation optimization 1 split}
\left.\begin{array}{ll}
\underset{u=(\uC,\uF,\uD)}{\text{Minimize}}\ &
\displaystyle{\frac{1}{2}u^\top A(\P,z^{k-1})u}\\
\text{ subject to}\ &\uC\in \Lambda_0:=\{u|\, u{\cdot}\vec{n}\geq 0\}\ \ \ \text{ and }\ \ \ \uD=\wD^k,
\end{array}\right\}
\end{equation}
where the components of $u=(\uC,\uF,\uD)$ correspond to the displacement on 
contact boundary $\GC$, in free nodes (interior and Neumann) in 
$\bar\Omega\setminus(\GC{\cup}\GD)$, and on 
Dirichlet boundary $\GD$, respectively. Matrix $A$ has the following form
$$
A(\P,z^{k-1}) = \bpm A\CC & A\CF & A\CD\\ A\FC& A\CF & A\FD\\ A\DC& A\DF & A\DD \epm + \bpm \At(\P,z^{k-1}) & 0 & 0\\ 0& 0 & 0\\ 0& 0 & 0 \epm
$$
where the first part corresponds to the discretization of the first part of the objective in \eqref{Label state equation optimization cont lower 1} and similarly for the second part. Using simple calculus, discretized problem \eqref{Label state equation optimization 1 split}
can be written as
\begin{subequations}\label{Label state equation optimization reduced}
\begin{equation}\label{Label state equation optimization reduced 1}
\left.\begin{array}{ll}
\underset{\uC}{\text{Minimize}}\ &
\displaystyle{\frac{1}{2}\uC^\top(A_\alpha+\At(\P,z^{k-1}))\uC
+(A_\beta \wD^k)^\top \uC}\\
\text{subject to}\ &\uC\in\Lambda_0,
\end{array}\right\}
\end{equation}
where we have defined
$$
\begin{array}{ll}
A_\alpha := A\CC-A\CF A\FF^{-1}A\FC, &\qquad A_\gamma := -A\FF^{-1}A\FC,\\
A_\beta := A\CD-A\CF A\FF^{-1}A\FD, &\qquad A_\delta := -A\FF^{-1}A\FD.
\end{array}
$$
Similarly, when discretizing 
\eqref{Label state equation optimization cont lower 2}, we obtain the 
following problem
\begin{equation}\label{Label state equation optimization reduced 2}
\left.\begin{array}{ll}
\underset{z}{\text{Minimize}}\ &
\displaystyle{\frac{1}{2}z^\top Bz+b(\P,u^k)^\top z}\\[.5em]
\text{subject to}\ &0\leq z\leq z^{k-1}.
\end{array}\right\}
\end{equation}
\end{subequations}

Since both problems in \eqref{Label state equation optimization reduced} are 
quadratic, we can pass to their necessary optimality conditions and the whole 
optimization problem \eqref{Label state equation inclusion original} 
reads as
\begin{align}\label{Label optimal control discretized}
\left.\begin{array}{ll}
\underset{\P,\uC,z}{\text{Minimize}}\!\! &\displaystyle{
\tau\sum_{k=1}^K\Big[\frac\zeta2\big|\uC^k-[u^{}_\mathrm{d}]_{\scriptscriptstyle\textrm{C}}^k\big|^2 + \frac\zeta2\big|A_\gamma \uC^k+A_\delta \wD^k
-[u^{}_\mathrm{d}]_{\scriptscriptstyle\textrm{F}}^k\big|^2 + 
\frac12\big|z^k-z_\mathrm{d}^k\big|^2\Big]}\\
\text{subject to}\!
&0\in (A_\alpha+\At(\P,z^{k-1}))\uC^k+A_\beta \wD^k+
\Nc_{\Lambda_0}(\uC^k),
\ \ k=1,...,K,\ \ u^0=u_0,\!
\\[.5em]
&0\in Bz^k+b(\P,u^k)+\Nc_{[0,z^{k-1}]}(z^k),
\ \ \ \ \ \ \ \ \ k=1,...,K,\ \ \ \ \ \ \ \ \ z^0=z_0,\\[.5em]
&\P\in \Pi.
\end{array}\right\}\!
\end{align}
By passing from $u^k$ to $\uC^k$ we have managed to reduce the number of 
parameters in \eqref{Label state equation optimization cont lower 1} from the 
number of all nodes to the number of contact nodes only. This is especially 
powerful because 
the first inclusion in \eqref{Label optimal control discretized}
 will be solved many times during the parameter identification procedure while it is sufficient to compute matrices $A_\alpha$, $A_\beta$, $A_\gamma$ and $A_\delta$ only once.

To be able to use Theorem \ref{Theorem noc}, we need to check whether 
assumptions (A1)--(A4) are satisfied. But this amount to showing that matrices 
$A_\alpha+\tilde{A}(\P,z^{k-1})$ and $B$ are positive definite. Since $A_\alpha$ 
is Schur complement of $A\CC$ in 
$\hat{A}:=\bpm A\CC\!\!&\!\!A\CF\\[-.2em] A\FC\!\!&\!\!A\FF\epm$, it is 
positive definite if $\hat{A}$ is positive definite. But the positive 
definiteness of $\hat{A}$ follows from the 
conformal FEM via positive definiteness of $\C$ together with the Korn
inequality using Dirichlet boundary conditions on $\GD$.
More precisely, the FEM may also involve some numerical integration
(which in fact has been used for our implementation, too).

\begin{remark}[Boundary-element method]
Note that \eqref{Label state equation optimization reduced 1} is the 
optimization problem on $\GC$ because we eliminated the values $\uD$ and $\uF$.
This is the philosophy of the \emph{boundary-integral method} and 
$(A_\beta\wD^k)^\top$ in \eqref{Label state equation optimization reduced 1}
is in the position of the (discretized) Poincar\'e-Steklov operator 
transferring Dirichlet boundary conditions on $\GC$ to traction forces
on $\GC$. 
The discretization then leads to the celebrated 
\emph{Boundary-Element Method} (BEM).
One option for this discretization is FEM, cf.\ e.g.\ \cite{LanSte07CFBE},
which is in fact what we used here and such BEM represents a noteworthy 
interpretation of \eqref{Label state equation optimization reduced 1}.
Other options are based on a direct discretization of the 
Poincar\'e-Steklov operator by using the approximate evaluation of 
the so-called Somigliana identity based on 
the underlying integral Green operators, cf.\ e.g.\ 
\cite{BVPM02CCDB,ParCan97BEMF,TMGP11BACO,SauSch11BEM}.
\end{remark}

\begin{remark}[Variants of the adhesive model]
The contribution $h(z)$ in \eqref{ansatz+E} has the meaning of a stored 
energy deposited in the adhesive bonds and, during delamination, this
energy naturally increases. 
If a reversible damage (called healing) were allowed, cf.\ 
Remark~\ref{rem-general} above, $h'(z)$ would give a driving force 
for it. Strict convexity of $h$ represents 
certain \emph{cohesive effects}: when delamination is tended to be complete,
still more and more energy is needed for complete delamination.
Cohesive effects can also be modelled by letting $\kappaN$ and 
$\kappaT$ dependent on $z$ so that $z\mapsto z\kappaN(z)$
and  $z\mapsto z\kappaT(z)$ are convex. This however does not 
guarantee strict convexity of $\Ecal(t,\P,u,\cdot)$.
Other option complying with a purely adhesive contact (e.g.\ $h=0$)
would be to consider a small, linear \emph{viscosity} in $z$, i.e.\ $\Rcal_2$
strictly convex and quadratic. Then the usual concept 
of weak solution can be used again together with the semi-implicit 
fractional-step-type time discretization. Yet, such problem becomes 
computational difficult if the viscosity is small, 
as often considered with the goal to approximate so-called vanishing-viscosity 
solution in the rate-independent inviscid limit, cf.\ \cite{roubicek.2013a}.
\end{remark}

\section{Numerical experiments}\label{sect-numerics}
In this section we illustrate usage and efficiency of the theory developed 
in Section \ref{sect-opt-cond} and later specified in 
Section~\ref{sect-formulation} on a two-dimensional problem where an 
elastic body glued along the $x$-axis and pulled
in the direction of the 
$y$-axis by the time-varying loading $\wD$, cf.\ Figure~\ref{Figure body}. 
\begin{figure}[!ht]
\begin{center}
\includegraphics[width=0.7\textwidth]{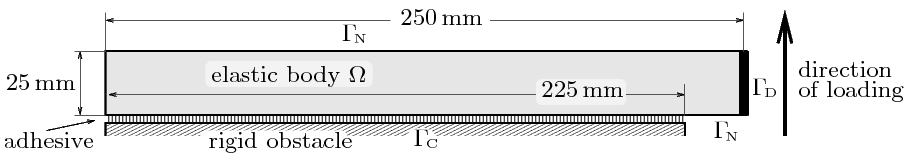}
\end{center}
\vspace*{-2em}
\caption{Geometry and boundary conditions of the two-dimensional
problem used for calculation.}
\label{Figure body}
\end{figure}
Considering the parameters $\alphaF$, $\kappaN$, and $\kappaT$ to be unknown, the main goal is to identify them via 
an inverse problem. Following the delamination example in 
\cite{roubicek.2013a,RoPaMa??LSQR}, we considered the isotropic material in 
the bulk with the tensor of elastic moduli
$$
\C_{ijkl}:=\frac{\nu E}{(1{+}\nu)(1{-}2\nu)}\delta_{ij}\delta_{kl}+\frac{E}{2(1{+}\nu)}(\delta_{ik}\delta_{jl}+\delta_{il}\delta_{jk})
$$
with the Young modulus $E = 70\text{ GPa}$ and the Poisson ratio $\nu=0.35$. Concerning the adhesive-stored-energy and the gradient terms in 
\eqref{ansatz+E}, we used 
$h(z)=\frac12z^2-z$ and $\eps=1\,$J,
while the weight $\zeta$ was chosen as $10^{10}$m$^{-2}$. For the space discretization we employed a mesh with $14\times 20$ nodes and 
the equidistant time discretization used $40$ time instants. The contact 
boundary consists of $12$ nodes. As already said, there are three parameters 
to be identified: $\alphaF$, $\kappaN$, and $\kappaT$. Moreover, we assume 
that the values of these parameters are not constant along the contact 
boundary but it may have different values in every contact node. This leads 
to a total number of $3\times 12=36$ parameters to be identified.

We fixed these $36$ values, to be more specific the mean of $\alphaF$, 
$\kappaN$, and $\kappaT$ was $187.5 \text{ J/m}^2$, $150\text{ GPa/m}$, and 
$75 \text{ GPa/m}$, respectively. The difference between the smallest and 
largest value of $\alphaF$ was approximately $10\%$ and similarly for $\kappaN$ and $\kappaT$. Next, we randomly generated some (with time increasing) dragging loading
 $\wD$, computed 
the corresponding $(u^{}_\mathrm{d},z^{}_\mathrm{d})$, and plugged them into 
the upper level of 
problem \eqref{Label optimal control discretized}. Since there was no 
perturbation of $(u^{}_\mathrm{d},z^{}_\mathrm{d})$ present, the optimal 
objective value was zero, which allows numerical testing of the 
efficiency of the optimization algorithm.

The computation of problem \eqref{Label optimal control discretized} was 
performed in Matlab. To compute $u^k$ from the first inclusion in 
\eqref{Label optimal control discretized},
we modified and used the already written code \cite{alberty.carstensen.funken.klose.2002}. Since a direct application of a gradient algorithm to whole problem \eqref{Label optimal control discretized} lead to rather inferior results, we had to find another way to solve \eqref{Label optimal control discretized}, specifically we used a combination of three optimization algorithms. The first was PSwarm \cite{vaz.vicente.2007}, which combines pattern search with genetic algorithm particle swarm, the second one standard Matlab function {\tt fminunc} and the last one a nonsmooth modification of BFGS algorithm \cite{lewis.overton.2012} with its implementation \cite{skajaa.2010}.

The optimization process was run in four phases. For the first phase, we simplified the problem and assumed that the parameters are constant along the contact boundary. This reduced the number of parameters from $36$ to $3$. To this problem, the algorithm PSwarm was used, however, we did not let it converge to the optimal solution but it was interrupted when the problem reached a priori given threshold or when the optimal value did not improve much in several successive iterations. In other words, the goal of the first phase was to find an estimate of the solution. Since PSwarm works rather with populations instead of single points, multiple initial points had to be chosen. These points were generated randomly from the following intervals
$$
\alphaF\in\big[100\text{ J/m}^2, 500\text{ J/m}^2\big],\qquad
\kappaN, \kappaT\in\big[10\text{ GPa/m}, 1000\text{ GPa/m}\big].
$$

In the second phase, the reduced problem was still considered but this time, an algorithm using a gradient information was used. Similarly to the first phase, we did not let the it converge and interrupted it prematurely. Because of this interruption, nonregular points were usually evaded and it was possible to use {\tt fminunc}, even though it is designed for smooth functions.

While in the first two phases, the values of parameters were constant on the contact boundary, this no longer holds true for the last two phases. In the third one, we considered the state in which one parameter corresponds to two nodes on the contact boundary, while in the fourth phase every parameter corresponded to only one node. This means that there were $18$ parameters in the third phase and $36$ in the last one. The evolution of the optimal value can be seen in Figure \ref{Figure convergence}. Note that on the $y$ axis the logarithm of the objective value is depicted and that the vertical lines separate the four phases.
\begin{figure}[!ht]
\begin{center}\includegraphics[width=0.7\textwidth]{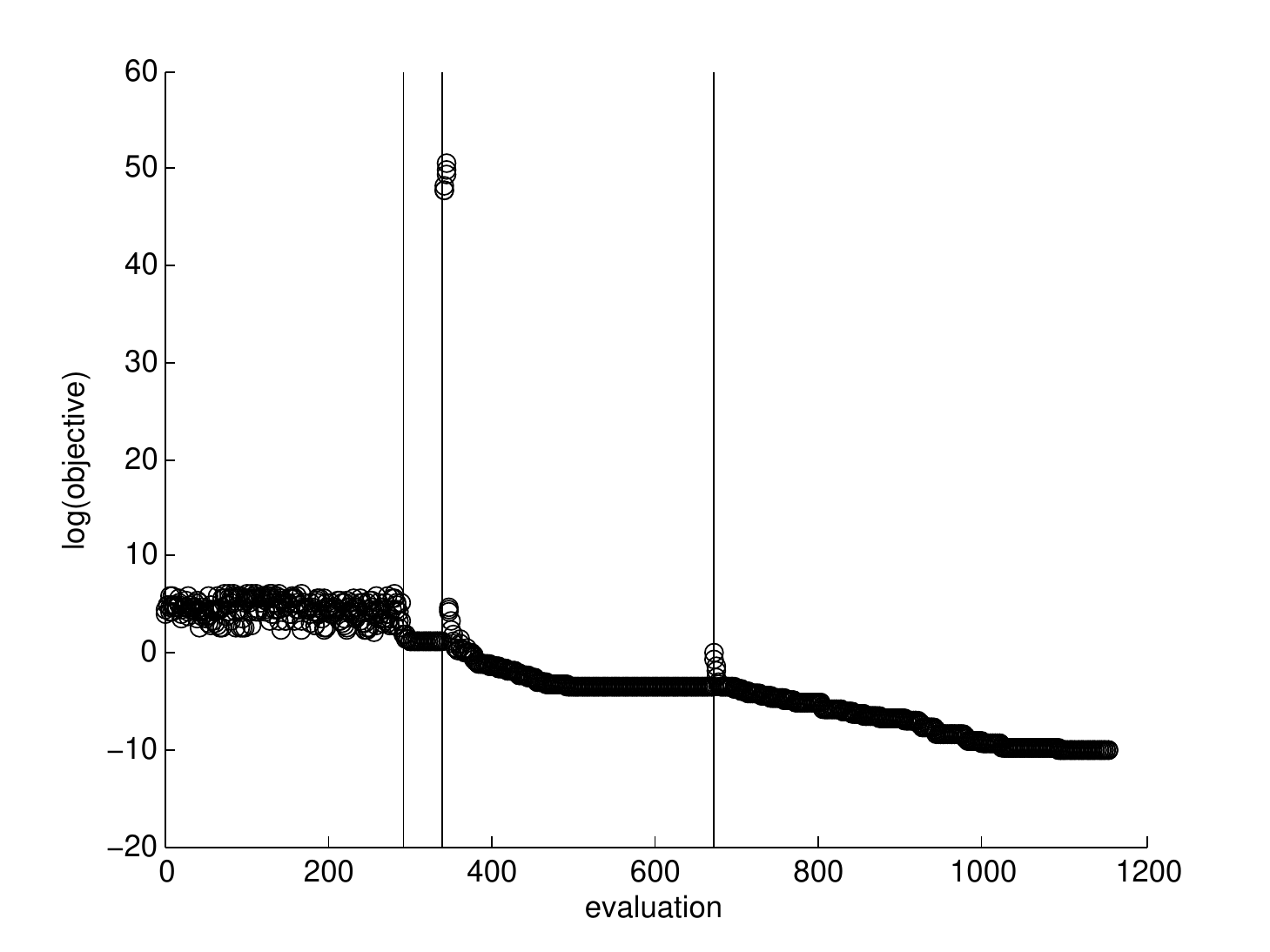}
\end{center}
\vspace*{-2em}
\caption{Development of the objective value during particular iterations of 
the optimization algorithms used during the four phases of our optimization: phase 1 used a global optimization algorithm (PSwarm), whereas phases 
2--4 used a (sub)gradient algorithm with subsequently refined discretization of $\GC$.}
\label{Figure convergence}
\end{figure}

The following table summarizes the values of parameters and of the objective function for all phases. The first column presents the best point in the initial population of PSwarm. The next four columns show the optimal solutions and values of all four phases. Finally, the last column corresponds to the actual values of parameters. Since there were multiple values distributed along the boundary for the last three columns, we show only their mean in such cases.
$$
\begin{array}{c||llllll}
&\text{starting} &\text{phase }1 &\text{phase }2 &\text{phase }3&\text{optimal}&\text{desired} \\\hline\hline
\alphaF &203.934 &190.405 &194.877 &187.489 &187.512 &187.5\\ 
\kappaN &0.822{\cdot}10^{11} &1.586{\cdot}10^{11} &1.462{\cdot}10^{11} &1.499{\cdot}10^{11} &1.500{\cdot}10^{11} &1.5{\cdot}10^{11}\\ 
\kappaT &47.251{\cdot}10^{10} &2.326{\cdot}10^{10} &7.317{\cdot}10^{10} &7.498{\cdot}10^{10} &7.499{\cdot}10^{10} &7.5{\cdot}10^{10}\\\hline 
\text{objective} &3138.97 &70.503 &14.184 &3.573{\cdot}10^{-4} &7.538{\cdot}10^{-11} &\ \ \ \ 0\\ 
\end{array} 
$$ 

In Figure \ref{Figure delamination} we show the the displacement $u$ (magnified by factor $50$) corresponding to one of the random initial points used for PSwarm and solution of the four phases. A circle on the contact boundary mean that no delamination has taken place yet at the corresponding node while an asterisk means that the corresponding node has been completely delaminated. No symbol being present indicates that only a partial delamination took place. Since the contact boundary is shorter than the length of the body, there are no symbols at the bottom right corner.
\begin{figure}[!ht]
\begin{center}
\includegraphics[width=0.95\textwidth,height=0.7\textwidth]{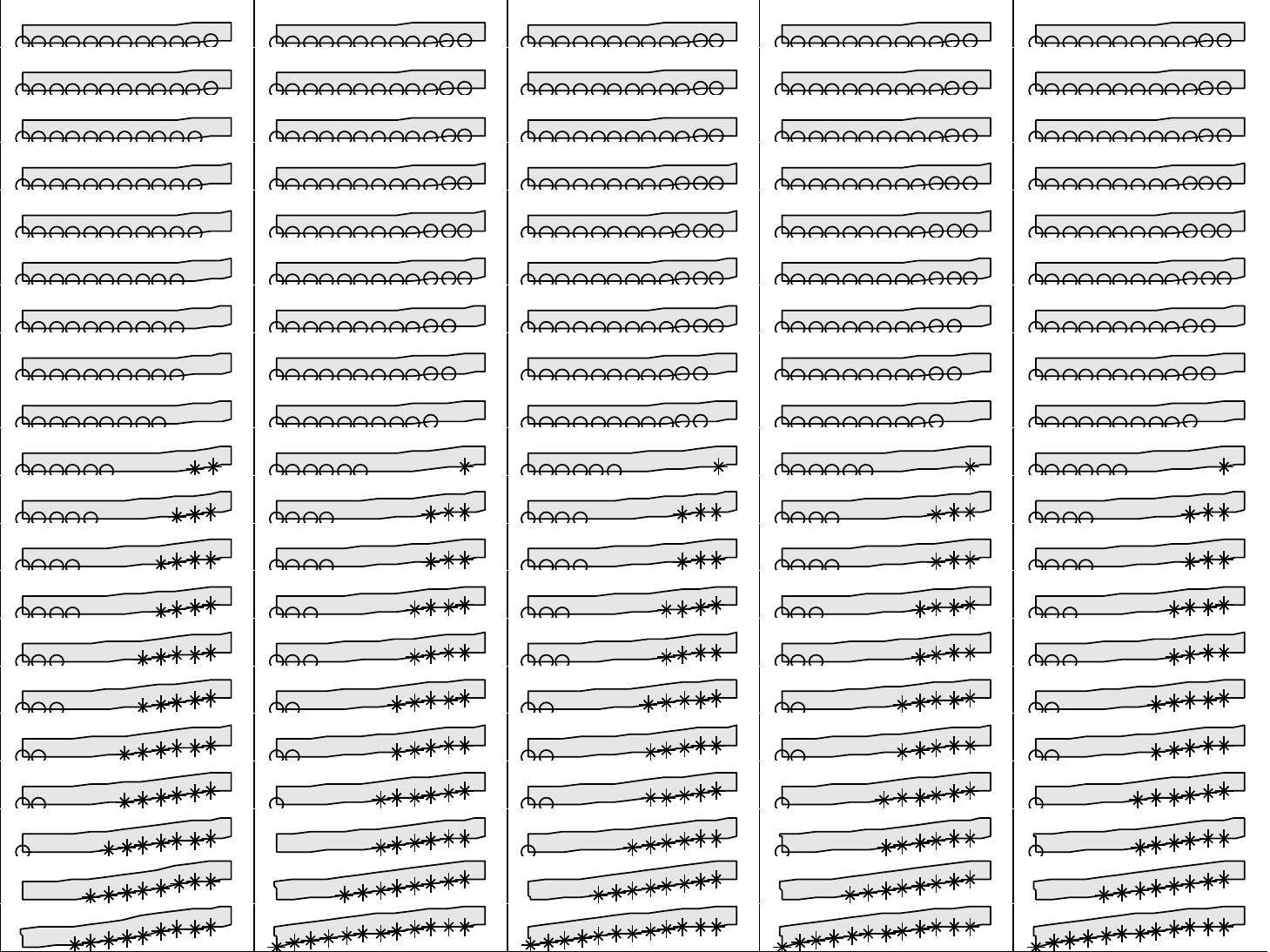}
\end{center}
\vspace*{-1em}
\caption{Evolution of the deformed specimen with distribution of the delamination parameter $z$ along $\GC$ (only values 1 or 0 are displayed) at 17 selected time instances. Displacements depicted as magnified by factor $50\times$.}
\label{Figure delamination}
\end{figure}

In Figure \ref{Figure parameters} we show the distribution of the elastic adhesive moduli $\kappaN$ and $\kappaT$ as well as the fracture toughness $\alphaF$ along the contact boundary $\GC$. Four lines corresponding to the actual parameters and to the terminal points of phases $2$, $3$ and $4$ are depicted. The horizontal line without any symbols corresponds to phase $2$, the line with circles corresponds to line $3$ and the line with asterisks corresponds to phase $4$, which means that it depicts the parameters identified by the algorithm. The last line without any symbols (which coincides with the line with asterisks for $\kappaN$) depicts the actual parameters. We see that 
while the result of phase $3$ provided a good estimate for the actual parameters. Phase $4$ provided only a slight improvement for $\kappaT$ while it managed to identify the values of $\kappaN$ completely.
\begin{figure}[!ht]
\begin{center}
\includegraphics[width=0.99\textwidth, height=0.35\textwidth]{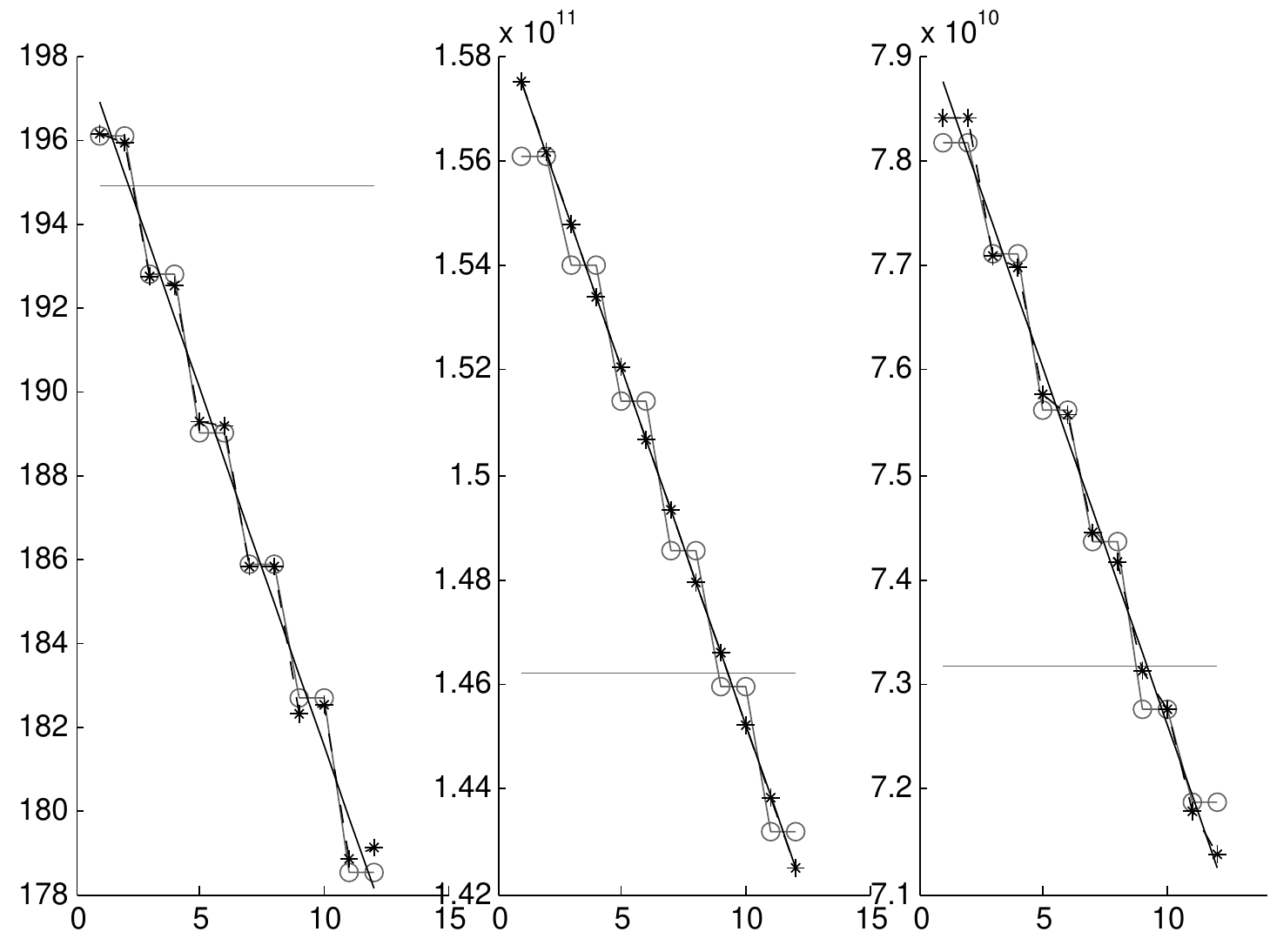}
\end{center}
\vspace*{-2em}
\caption{Parameter distribution along the contact boundary, graphs 
depicting form left to right $\alphaF$, $\kappaN$ and $\kappaT$
resulting after particular phases of the optimization algorithm.
}
\label{Figure parameters}
\end{figure}

\section*{Concluding remarks}
The optimality conditions stated in Theorem \ref{Theorem noc} are in the MPEC 
literature called M-stationarity conditions because they are based exclusively 
on notions from the Mordukhovich subdifferential calculus. They are relatively 
sharp and can very well be used, e.g., for testing this type of stationarity 
at points computed via ImP. It would be a great challenge to derive suitable 
optimality conditions also for the original  continuous MPEEC 
\eqref{identification}. Unfortunately, this problem is formulated over 
non-Asplund spaces (with a possible exception of the space for the variable 
$\P$) which are not amenable for a treatment via the Mordukhovich calculus.

If function $\Rcal_2$ in \eqref{identification-disc} happens to be 
nondifferentiable as in Remarks \ref{rem-general}--\ref{rem-general-II},
then in the generalized equation system defining $S$ one has to do with sums of 
multifunctions. Such situations occur in the so-called Stampacchia variational 
inequalities, whose sensitivity and stability analysis represents a great open 
problem. One possible way to overcome this hurdle could be a smoothening of 
$\Rcal_2$ or a smooth penalization of the constraint.

In this paper, apart from some particular cases like that mentioned in 
Remark~\ref{rem-convergence}, the main peculiarity consists in the fact that 
the original continuous problem in \eqref{identification} (whose parameters 
$\pi$ are to be identified) does not need to have a unique 
response. To control (or identify) such systems is, from very fundamental
reasons, very doubtful. Therefore, one should interpret the approach used in
this article carefully, counting only with a response of \eqref{Biot}
that can be approximated by a particular numerical method 
\eqref{identification-disc} and being aware that possibly some other
responses might exist and give even more accurate results. This is 
the best one can do and assume in identification of system of 
the type \eqref{Biot} which may cover very general situation 
e.g.\ sudden ruptures which are naturally very difficult to be
controlled (or identified).

{\small
\section*{Acknowledgments}
This research has been supported by GA\,\v CR through
the projects 201/12/0671 
``Variational and numerical analysis in nonsmooth continuum mechanics'',
201/10/0357 ``Modern mathematical and computational 
models for inelastic processes in solids'', 13-18652S ``Computational modeling of damage and transport processes 
in quasi-brittle materials'', 14-15264S 
``Experimentally justified multiscale modelling of shape memory alloys''
with the institutional support RVO:61388998 (\v CR) and SVV--2014--260 105.
}

\bibliographystyle{abbrv}
\bibliography{lajvotr-opt-delam}

\end{document}